\documentclass[12pt]{article}
\usepackage{amsthm, amsmath, amssymb, enumerate}
\usepackage{fullpage}
\usepackage[colorlinks=true]{hyperref}
\usepackage[all]{xy}
\begin{document}
\title{Explicit Kodaira--Spencer maps over PEL Shimura varieties}
\author{Ziqi Guo}
\date{}
\maketitle

\theoremstyle{plain}
\newtheorem{thm}{Theorem}[section]
\newtheorem{theorem}[thm]{Theorem}
\newtheorem{cor}[thm]{Corollary}
\newtheorem{corollary}[thm]{Corollary}
\newtheorem{lem}[thm]{Lemma}
\newtheorem{lemma}[thm]{Lemma}
\newtheorem{pro}[thm]{Proposition}
\newtheorem{proposition}[thm]{Proposition}
\newtheorem{prop}[thm]{Proposition}
\newtheorem{definition}[thm]{Definition}
\newtheorem{assumption}[thm]{Assumption}

\theoremstyle{remark} 
\newtheorem{remark}[thm]{Remark}
\newtheorem{example}[thm]{Example}
\newtheorem{remarks}[thm]{Remarks}
\newtheorem{problem}[thm]{Problem}
\newtheorem{exercise}[thm]{Exercise}
\newtheorem{situation}[thm]{Situation}
\newtheorem{acknowledgment}[thm]{Acknowledgment}

\numberwithin{equation}{subsection}

\newcommand{\ZZ}{\mathbb{Z}}
\newcommand{\CC}{\mathbb{C}}
\newcommand{\QQ}{\mathbb{Q}}
\newcommand{\RR}{\mathbb{R}}
\newcommand{\HH}{\mathcal{H}}     

\newcommand{\ad}{\mathrm{ad}}            
\newcommand{\NT}{\mathrm{NT}}         
\newcommand{\nonsplit}{\mathrm{nonsplit}}         
\newcommand{\Pet}{\mathrm{Pet}}         
\newcommand{\Fal}{\mathrm{Fal}}         

\newcommand{\cs}{{\mathrm{cs}}}         

\newcommand{\ZZn}{\mathcal{O}_{E_0}[\frac{1}{N}]}         
\newcommand{\ZZN}{\mathbb{Z}[\frac{1}{N}]}     
\newcommand{\XK}{X_K}    
\newcommand{\XXK}{\mathcal{X}_K}    
\newcommand{\OA}{\underline{\Omega}_\mathcal{A}}
\newcommand{\OU}{\Omega_{\mathcal{X}_K/\mathcal{O}_{E_0}[\frac{1}{N}]}}
\newcommand{\WA}{\underline{\omega}_\mathcal{A}}
\newcommand{\WU}{\omega_{\mathcal{X}_K/\mathcal{O}_{E_0}[\frac{1}{N}]}}
\newcommand{\HHom}{\mathcal{H}\mathrm{om}}
\newcommand{\Hr}{\mathcal{H}_{\frac{r}{2},\frac{r}{2}}}
\newcommand{\OBr}{\mathcal{O}_B^\frac{r}{n}}

\newcommand{\pair}[1]{\langle {#1} \rangle}
\newcommand{\wpair}[1]{\left\{{#1}\right\}}
\newcommand{\wh}{\widehat}
\newcommand{\wt}{\widetilde}

\newcommand\Spf{\mathrm{Spf}}

\newcommand{\lra}{{\longrightarrow}}

\newcommand{\matrices}[2]
{\left( \begin{array}{c}
  #1   \\
  #2   \\
 \end{array}\right)} 

\newcommand{\matrixx}[4]
{\left( \begin{array}{cc}
  #1 &  #2  \\
  #3 &  #4  \\
 \end{array}\right)}        


\newcommand{\BA}{{\mathbb {A}}}
\newcommand{\BB}{{\mathbb {B}}}
\newcommand{\BC}{{\mathbb {C}}}
\newcommand{\BD}{{\mathbb {D}}}
\newcommand{\BE}{{\mathbb {E}}}
\newcommand{\BF}{{\mathbb {F}}}
\newcommand{\BG}{{\mathbb {G}}}
\newcommand{\BH}{{\mathbb {H}}}
\newcommand{\BI}{{\mathbb {I}}}
\newcommand{\BJ}{{\mathbb {J}}}
\newcommand{\BK}{{\mathbb {K}}}
\newcommand{\BL}{{\mathbb {L}}}
\newcommand{\BM}{{\mathbb {M}}}
\newcommand{\BN}{{\mathbb {N}}}
\newcommand{\BO}{{\mathbb {O}}}
\newcommand{\BP}{{\mathbb {P}}}
\newcommand{\BQ}{{\mathbb {Q}}}
\newcommand{\BR}{{\mathbb {R}}}
\newcommand{\BS}{{\mathbb {S}}}
\newcommand{\BT}{{\mathbb {T}}}
\newcommand{\BU}{{\mathbb {U}}}
\newcommand{\BV}{{\mathbb {V}}}
\newcommand{\BW}{{\mathbb {W}}}
\newcommand{\BX}{{\mathbb {X}}}
\newcommand{\BY}{{\mathbb {Y}}}
\newcommand{\BZ}{{\mathbb {Z}}}

\newcommand{\CA}{{\mathcal {A}}}
\newcommand{\CB}{{\mathcal {B}}}
\newcommand{\CD}{{\mathcal{D}}}
\newcommand{\CE}{{\mathcal {E}}}
\newcommand{\CF}{{\mathcal {F}}}
\newcommand{\CG}{{\mathcal {G}}}
\newcommand{\CH}{{\mathcal {H}}}
\newcommand{\CI}{{\mathcal {I}}}
\newcommand{\CJ}{{\mathcal {J}}}
\newcommand{\CK}{{\mathcal {K}}}
\newcommand{\CL}{{\mathcal {L}}}
\newcommand{\CM}{{\mathcal {M}}}
\newcommand{\CN}{{\mathcal {N}}}
\newcommand{\CO}{{\mathcal {O}}}
\newcommand{\CP}{{\mathcal {P}}}
\newcommand{\CQ}{{\mathcal {Q}}}
\newcommand{\CR }{{\mathcal {R}}}
\newcommand{\CS}{{\mathcal {S}}}
\newcommand{\CT}{{\mathcal {T}}}
\newcommand{\CU}{{\mathcal {U}}}
\newcommand{\CV}{{\mathcal {V}}}
\newcommand{\CW}{{\mathcal {W}}}
\newcommand{\CX}{{\mathcal {X}}}
\newcommand{\CY}{{\mathcal {Y}}}
\newcommand{\CZ}{{\mathcal {Z}}}

\newcommand{\ab}{{\mathrm{ab}}}
\newcommand{\Ad}{{\mathrm{Ad}}}
\newcommand{\an}{{\mathrm{an}}}
\newcommand{\Aut}{{\mathrm{Aut}}}

\newcommand{\Br}{{\mathrm{Br}}}
\newcommand{\bs}{\backslash}
\newcommand{\bbs}{\|\cdot\|}

\newcommand{\Ch}{{\mathrm{Ch}}}
\newcommand{\cod}{{\mathrm{cod}}}
\newcommand{\cont}{{\mathrm{cont}}}
\newcommand{\cl}{{\mathrm{cl}}}
\newcommand{\criso}{{\mathrm{criso}}}
\newcommand{\de}{{\mathrm{d}}}
\newcommand{\dR}{{\mathrm{dR}}}
\newcommand{\df}{\mathrm{det}^*}
\newcommand{\disc}{{\mathrm{disc}}}
\newcommand{\Disc}{\mathrm{Disc}}
\newcommand{\Div}{{\mathrm{Div}}}
\renewcommand{\div}{{\mathrm{div}}}

\newcommand{\Eis}{{\mathrm{Eis}}}
\newcommand{\End}{{\mathrm{End}}}

\newcommand{\Frob}{{\mathrm{Frob}}}

\newcommand{\Gal}{{\mathrm{Gal}}}
\newcommand{\GL}{{\mathrm{GL}}}
\newcommand{\GO}{{\mathrm{GO}}}
\newcommand{\GSO}{{\mathrm{GSO}}}
\newcommand{\GSp}{{\mathrm{GSp}}}
\newcommand{\GSpin}{{\mathrm{GSpin}}}
\newcommand{\GU}{{\mathrm{GU}}}
\newcommand{\BGU}{{\mathbb{GU}}}

\newcommand{\Hom}{{\mathrm{Hom}}}
\newcommand{\Hol}{{\mathrm{Hol}}}
\newcommand{\HC}{{\mathrm{HC}}}
\newcommand{\Herm}{\mathrm{Herm}}
\newcommand{\id}{\mathrm{id}}
\newcommand{\Img}{{\mathrm{Im}}}
\newcommand{\Imh}{\mathrm{Imh}}
\newcommand{\Ind}{{\mathrm{Ind}}}
\newcommand{\inv}{{\mathrm{inv}}}
\newcommand{\Isom}{{\mathrm{Isom}}}

\newcommand{\Jac}{{\mathrm{Jac}}}
\newcommand{\JL}{{\mathrm{JL}}}

\newcommand{\Ker}{{\mathrm{Ker}}}
\newcommand{\KS}{{\mathrm{KS}}}

\newcommand{\Lie}{{\mathrm{Lie}}}

\newcommand{\M}{\mathrm{M}}

\newcommand{\new}{{\mathrm{new}}}
\newcommand{\Nm}{\mathrm{Nm}}
\newcommand{\NS}{{\mathrm{NS}}}

\newcommand{\ord}{{\mathrm{ord}}}
\newcommand{\ol}{\overline}
\newcommand{\otf}{\otimes^*}
\newcommand{\rank}{{\mathrm{rank}}}

\newcommand{\PGL}{{\mathrm{PGL}}}
\newcommand{\PSL}{{\mathrm{PSL}}}
\newcommand{\Pic}{\mathrm{Pic}}
\newcommand{\Prep}{\mathrm{Prep}}
\newcommand{\Proj}{\mathrm{Proj}}

\newcommand{\Picc}{\mathcal{P}ic}

\renewcommand{\Re}{{\mathrm{Re}}}
\newcommand{\Reh}{\mathrm{Reh}}
\newcommand{\Res}{{\mathrm{Res}}}
\newcommand{\red}{{\mathrm{red}}}
\newcommand{\reg}{{\mathrm{reg}}}
\newcommand{\sm}{{\mathrm{sm}}}
\newcommand{\sing}{{\mathrm{sing}}}
\newcommand{\SL}{\mathrm{SL}}
\newcommand{\Sp}{\mathrm{Sp}}
\newcommand{\Sym}{{\mathrm{Sym}}}

\newcommand{\tor}{{\mathrm{tor}}}
\newcommand{\tr}{{\mathrm{tr}}}

\newcommand{\ur}{{\mathrm{ur}}}
\newcommand{\U}{\mathrm{U}}

\newcommand{\vol}{{\mathrm{vol}}}

\newcommand{\ds}{\displaystyle}

\begin{abstract}
The goal of our work is to construct morphisms relating the Hodge line bundle and the relative dualizing sheaf on integral models of PEL Shimura varieties via Kodaira--Spencer maps, and explicitly compute these morphisms and their effects on the canonical metrics. This result provides a concrete method for comparing the corresponding arithmetic intersection numbers. In the proper case, it gives an explicit relationship between the modular heights defined by these bundles.

\end{abstract}

\tableofcontents

\section{Introduction}
The goal of this paper is to explicitly compute the Kodaira--Spencer map over a PEL Shimura variety and its effect on the metrics of line bundles. Our main work is to construct a canonical morphism between powers of the Hodge bundle and the canonical bundle using the classical Kodaira--Spencer map, and to give an explicit formula for its extension to an integral model. Our result contains two parts: one is the computation of images over integral models, the other is the comparison of metrics on the complex fibers. The definition of the Kodaira--Spencer map is canonical and easy to understand, but it turns out to be difficult to compute the explicit constants since there are various abstract identifications.

Using the main theorem of this paper, we explicitly compute the difference between the modular height formulas on PEL Shimura varieties associated with these two line bundles. This paper extends the explicit computations in \cite{Yuan2} and \cite{Guo1} to PEL Shimura varieties of types A and C. In \cite{Yuan2}, an explicit formula for the Kodaira--Spencer map is given for quaternionic Shimura curves over $\QQ$, while \cite{Guo1} extends these results to Hilbert--Siegel modular varieties and twisted Hilbert modular varieties.

\subsection{Main results}
We first introduce the integral model and the two metrized bundles occurring in the main theorems.

Let $(B,*,V,\langle\cdot,\cdot\rangle,h)$ be a global PEL datum as in Definition \ref{PEL datum}. Let $G/\QQ$ be the algebraic group determined by the functor
\begin{equation*}
    R\longmapsto\{x\in\End_{B\otimes R}(V\otimes R)|xx^*\in R^\times\},
\end{equation*}
and $X$ be the $G(\RR)$-conjugacy class of $h$. Then the pair $(G,X)$ is a Shimura datum, and for a compact open subgroup $K\subset G(\BA_f)$, the associated Shimura variety has complex points
\begin{equation*}
    X_K(\CC)=G(\QQ)\backslash(X\times G(\BA_f))/K.
\end{equation*}
Denote by $E_0$ the reflex field and by $X_K$ the canonical model over $E_0$. For the main theorems, we assume that $B$ is a central division algebra over a number field $F$.

Moreover, let $(\mathcal{O},*,L,\langle\cdot,\cdot\rangle,h)$ be a PEL-type $\mathcal{O}$-lattice as in Definition \ref{PEL type O lattice}, which is an integral version of the PEL datum. We assume that $\mathcal O$ is a $*$-stable maximal order and that $L$ is free over $\mathcal O$ and self-dual for $\langle\cdot,\cdot\rangle$. We also assume that $F/\QQ$ is unramified at every rational prime below a nonsplit place of $B$. Take $K=\prod_pK_p$ inside the stabilizer of $L\otimes\widehat\ZZ$, with maximal level at every odd prime below a nonsplit place of $B$ and sufficiently small level away from these primes that the moduli spaces are schemes. Let $m$ be the product of the primes $p$ for which $K_p$ is different from the subgroup preserving $L\otimes\ZZ_p$; see Remark \ref{maximal remark}. Put
\begin{equation*}
    N=2m|\operatorname{disc}(F)|.
\end{equation*}
We use the integral model $\mathcal X_K$ over $\ZZn$ constructed in Section \ref{Moduli interpretation and integral model} from the self-dual periodic multichain determined by $L$. It is flat, normal and Cohen--Macaulay, with geometrically reduced special fibers, and its relative smooth locus contains every codimension-one point. The member at $L$ is the universal abelian scheme $\pi:\mathcal A\to\mathcal X_K$, with principal polarization, $\mathcal O$-action and level structure $(\lambda,i,\eta_K)$ as defined in that section.

Write $\epsilon$ for its zero section and put
\begin{equation*}
    \underline\Omega_{\mathcal A}=\epsilon^*\Omega_{\mathcal A/\mathcal X_K},\qquad
    \underline\omega_{\mathcal A}=\det\underline\Omega_{\mathcal A}.
\end{equation*}
The relative dualizing sheaf $\WU$ is reflexive of rank one. On the complex fibers, the Hodge line bundle and the canonical bundle carry the Faltings and Petersson metrics, respectively. The Kodaira--Spencer map, defined by the connecting morphism of the differential sequence in Section \ref{Kodaira Spencer map section}, gives the determinant morphisms between these bundles that we compute below.

For the convenience of stating our main theorem, write $n$ for the degree of $B$ over $F$, $F^+$ for the totally real field fixed by $*$, and $g=[F^+:\QQ]$. Albert's classification gives types A, C and D as in Definition \ref{type of PEL Shimura variety}; we treat types A and C. Morita equivalence explains the restriction to a division algebra at the level of PEL data; see Remarks \ref{assumption remark} and \ref{Morita equivalence remark}. The integral computations use the fixed lattice and polarization above.

In the case of type C, $F=F^+$ and $B_\tau\cong\M_n(\RR)$ for any archimedean place $\tau:F\rightarrow\RR$; in the case of type A, $F/F^+$ is a totally imaginary quadratic extension and $B_\tau\cong\M_n(\CC)$ for any archimedean place $\tau:F\rightarrow\CC$. Suppose the symplectic $B$-module $V$ has dimension $2rn$ over $F^+$. Under these notations, the universal abelian scheme has relative dimension $l=rng$. Let $d_B$ be an ideal of $\mathcal{O}_F$ such that
\begin{equation*}
    d_B=\prod_{v\ \nonsplit\ \mathrm{in}\ B}v^{n_v}.
\end{equation*}
Here $v$ denotes a finite place of $F$, while we say $v$ is nonsplit (or ramified) in $B/F$ if
\begin{equation*}
    B_v:=B\otimes_F F_v\cong\M_{n_v}(D_v)
\end{equation*}
for some division algebra $D_v\ne F_v$ and positive integer $1\le n_v<n$. In both cases, $d_B$ is extended from an ideal of $\mathcal O_{F^+}$. We use the same notation for this ideal and put $\Nm(d_B)=\Nm_{F^+/\QQ}(d_B)$.

The self-duality of $L$ specifies the principal polarization in the moduli interpretation. By our assumptions, $F/\QQ$ is unramified and $K_p$ is maximal at every odd prime $p$ dividing $d_B$. Thus these primes do not divide $N$ and remain in the base $\ZZn$. The integral image at these primes is part of the main theorems.

The determinant over the center of $B$ contributes a constant Hermitian line bundle $\mathcal C$, defined in \eqref{constant determinant bundle} using the determinant of $\mathcal O_{F^+}$ and, in type A, that of $\mathcal O_F$ over $\mathcal O_{F^+}$. We use the same notation for its pull-backs. The trace pairings give its metric and an isometric trivialization of $\mathcal C^{\otimes2}$; when $n$ is even, they already trivialize $\mathcal C$.

For a holomorphic top form $\alpha$ on a complex fiber $\mathcal A_x$ of dimension $l=rng$, the Faltings metric is
\begin{equation*}
    \lVert\alpha\rVert_{\Fal}^2=\frac{1}{(2\pi)^l}\left|\int_{\mathcal A_x(\CC)}\alpha\wedge\bar\alpha\right|.
\end{equation*}
In type C the symmetric domain is a product of Siegel upper half-spaces $\HH_r$; in type A with balanced signature it is a product of Hermitian upper half-spaces $\Hr$. Let $Z_i$ be their matrix coordinates, let $Y_i$ be the positive definite imaginary parts, and let $\de\tau$ be the exterior product of the independent differentials. With the coordinate conventions of Section \ref{Kodaira Spencer map section}, the Petersson metrics are
\begin{equation*}
    \lVert\de\tau\rVert_{\Pet}=\begin{cases}
    2^{gr(r+1)/2}\displaystyle\prod_{i=1}^g\det(Y_i)^{(r+1)/2},&\text{in type C},\\[2mm]
    2^{gr^2/4}\displaystyle\prod_{i=1}^g\det(Y_i)^{r/2},&\text{in type A}.
    \end{cases}
\end{equation*}

The image over the integral model is described by an effective vertical Weil divisor
\begin{equation}\label{vertical divisor D}
    D=\sum_Z e_ZZ,
\end{equation}
where $Z$ runs through the irreducible components of the special fibers. The coefficients $e_Z$ are given by the following formulas. Propositions \ref{explicit basis of R} and \ref{type C local calculation} show that the coefficient $e_Z$ of $Z$ in $D$ is the order of vanishing along $Z$ of the determinant morphism $\psi$ in Theorems \ref{Intro main1} and \ref{Intro main2}. More precisely, at the generic point of $Z$, take an unramified extension separating the center embeddings and splitting the local division algebras, and denote the resulting discrete valuation ring by $\mathcal O$, with uniformizer $\pi$. Then $e_Z$ is the sum of the following orders over the embeddings of $F^+$; in type A, each conjugate pair is counted once.

In type C, each place $v$ where $B_v$ is split contributes zero to $e_Z$. At a quaternionic division place, the maximal order becomes $\left(\begin{smallmatrix}\mathcal O&\mathcal O\\\pi \mathcal O&\mathcal O\end{smallmatrix}\right)$ over $\mathcal O$. Write the Hodge submodule as $P_1^m\oplus P_2^{r-m}$, where $P_1=\mathcal O\oplus\pi\mathcal O$ and $P_2=\mathcal O^2$ are the two column modules of the local order. If the valuation of the discriminant of the orthogonal involution is odd, its contribution is
\begin{equation*}
    \binom{m+1}{2}+\binom{r-m+1}{2}.
\end{equation*}
If that valuation is even, the polarization forces $r$ to be even and $m=r/2$, and the contribution is $r^2/4$.

In type A, write $B_v=\mathrm M_{n_v}(D_v)$ as above, allowing $n_v=n$ at a split place. The column modules of the hereditary order obtained from $\mathcal O_{D_v}$ are
\begin{equation*}
    P_i=\mathcal O^i\oplus\pi \mathcal O^{n/n_v-i},\qquad 1\leq i\leq n/n_v.
\end{equation*}
After integral Morita equivalence, write one of the two Hodge factors as $\bigoplus_iP_i^{m_i}$. Its multiplicities satisfy $0\leq m_i\leq rn_v/n$ and $\sum_i m_i=r/2$. The contribution of this conjugate pair to $e_Z$ is
\begin{equation*}
    \frac{(n-2n_v)r^2}{8}+\frac n2\sum_{i=1}^{n/n_v}m_i^2.
\end{equation*}
It is independent of the factor chosen in the pair and is zero when $B_v$ is split.

Now we present the main results of this paper in our two cases.
\begin{theorem}[Type C]\label{Intro main1}
Suppose that $B$ is of type C. There is a canonical injection
\begin{equation*}
    \psi:\underline{\omega}_\mathcal A^{\otimes(r+1)}\otimes\mathcal C^{-1}\longrightarrow(\WU^{\otimes n})^{\vee\vee},
\end{equation*}
and its image is $(\WU^{\otimes n}\otimes\mathcal O_{\mathcal X_K}(-D))^{\vee\vee}$. Moreover, $\psi$ is an isometry for the Faltings metric, the trace determinant metric on $\mathcal C$, and the $n$-th power of the Petersson metric. Here $\mathcal C$ is the constant Hermitian line bundle introduced above, $\vee$ denotes the dual over $\mathcal O_{\mathcal X_K}$, $\vee\vee$ denotes the double dual, and $D$ is the effective vertical Weil divisor in \eqref{vertical divisor D}.
\end{theorem}

\begin{theorem}[Type A]\label{Intro main2}
Suppose that $B$ is of type A, and the signature of $V$ at each archimedean place is $(\frac r2,\frac r2)$. There is a canonical injection
\begin{equation*}
    \psi:\underline{\omega}_\mathcal A^{\otimes r/2}\otimes\mathcal C^{-1}\longrightarrow(\WU^{\otimes n})^{\vee\vee},
\end{equation*}
and its image is $(\WU^{\otimes n}\otimes\mathcal O_{\mathcal X_K}(-D))^{\vee\vee}$. Moreover, $\psi$ is an isometry for the Faltings metric, the trace determinant metric on $\mathcal C$, and the $n$-th power of the Petersson metric. Here $\mathcal C$ is the same constant Hermitian line bundle, and $D$ is given by \eqref{vertical divisor D}.
\end{theorem}

Note that if $\mathcal X_K$ is Gorenstein over $\ZZn$, then both double duals in the two theorems can be omitted. This applies in particular to semistable integral models. For instance, when $g=r=1$ and $n=2$, the Shimura variety is a quaternionic Shimura curve over $\QQ$. The two possible integral Hodge submodules at a division place both contribute $1$, so $D=\operatorname{div}(d_B)$ over $\ZZn$. Thus Theorem \ref{Intro main1} agrees with the main theorem in \cite{Yuan2}. When $n=1$ (resp. $n=2$ and $r=1$), it gives the Hilbert--Siegel (resp. twisted Hilbert modular) calculation of \cite{Guo1}, with the determinant over the center included.

For larger $r$, different integral Hodge submodules can give different orders. For example, in the odd-discriminant case of type C with $r=2$, the modules $P_1^2$ and $P_2^2$ both give $3$, whereas $P_1\oplus P_2$ gives $2$. Thus a constant does not describe all the allowed local types. The coefficient depends only on $m(r-m)$ in this case, and on $\sum_i m_i^2$ in type A, so different integral types may still give the same value. On the smooth locus of each special fiber, the coefficients $e_Z$ are locally constant, though they may differ between irreducible components. 

When $n$ is even, the trivialization of $\mathcal C$ gives the maps between the original powers of the two bundles. In all cases, taking tensor squares gives an untwisted morphism with divisor $2D$. The metric assertions concern the same morphism $\psi$, including its coefficient in period coordinates. The period-lattice volume and the polarization give the corresponding factors in the Faltings norm, as computed in Section \ref{sec metric}.

\subsection{Application to modular heights}\label{Application to modular heights}
As stated at the beginning of this section, the main application is a concrete method for comparing the arithmetic intersection numbers associated with the two bundles. For this application we assume that $\mathcal X_K$ is proper over $\ZZn$. We extend the bundles to a projective model over $\mathcal O_{E_0}$ as in Section \ref{subsec proper height comparison}; these extensions specify the metrics at the primes dividing $N$.

If $X_K$ has dimension $d$, the modular height for a Hermitian line bundle $\overline\omega$ with positive generic degree is normalized by
\begin{equation*}
    h_{\overline\omega}(X_K)=\frac{\widehat{\deg}(\hat c_1(\overline\omega)^{d+1})}{(d+1)[E_0:\QQ]\deg(c_1(\omega_{E_0})^d)}.
\end{equation*}
For the canonical bundle, we use the $n$-th root of the model metric constructed from $(\WU^{\otimes n})^{\vee\vee}$, together with the Petersson metric. Corollary \ref{cor proper height exact} expresses the difference between its height and the Hodge height as mixed intersections with the vertical Cartier divisor defined by $\psi$.

In particular, if the Cartier divisor $\widetilde D$ induced by $D$ defined in Section \ref{subsec proper height comparison} is $\operatorname{div}(c)$ for some positive rational number $c$, then the formulas take the form
\begin{equation*}
    \begin{aligned}
    n h_\omega(X_K)&=(r+1)h_{\underline\omega_{\mathcal A}}(X_K)+\log c-\frac{\widehat{\deg}(\overline{\mathcal C})}{[E_0:\QQ]}&&\text{in type C},\\
    n h_\omega(X_K)&=\frac r2h_{\underline\omega_{\mathcal A}}(X_K)+\log c-\frac{\widehat{\deg}(\overline{\mathcal C})}{[E_0:\QQ]}&&\text{in type A}.
    \end{aligned}
\end{equation*}
Here $h_{\underline\omega_{\mathcal A}}$ and $h_\omega$ use the Faltings and Petersson metrics at the infinite places, respectively. The model metrics at finite places away from $N$ are determined by our integral construction, while those at places dividing $N$ are specified by the extensions chosen in Section \ref{subsec proper height comparison}. 

Note that both corrections are explicit in familiar cases. For $\mathcal C$, we use the determinant modules in \eqref{constant determinant bundle} over the full rings of integers and equip them with the trace determinant metrics. These choices give
\begin{equation}\label{degree of constant determinant bundle}
    \frac{\widehat{\deg}(\overline{\mathcal C})}{[E_0:\QQ]}=
    \begin{cases}
    -\dfrac{nr(r+1)}4\log|\operatorname{disc}(F^+)|,&\text{in type C},\\[2mm]
    \begin{aligned}
    &-\dfrac{nr^2}{8}\log|\operatorname{disc}(F^+)|\\
    &-\dfrac{nr(r+2)}{16}\log\Nm_{F^+/\QQ}\bigl(\operatorname{disc}(F/F^+)\bigr),
    \end{aligned}&\text{in type A}.
    \end{cases}
\end{equation}
Here $\operatorname{disc}(F/F^+)$ is the relative discriminant ideal. Indeed, an integral determinant module with the trace determinant metric has arithmetic degree minus one half the logarithm of the corresponding discriminant. For example, in the twisted Hilbert modular case $n=2$, $r=1$ and $F^+=\QQ(\sqrt5)$, the basis $1,(1+\sqrt5)/2$ has trace matrix $\left(\begin{smallmatrix}2&1\\1&3\end{smallmatrix}\right)$, of determinant $5$. Hence the normalized degree of $\overline{\mathcal C}$ is $-\log5$. 

For a quaternionic Shimura curve over $\QQ$, take $g=r=1$, $n=2$, and use the integral model and metrics of \cite[Theorem 1.1]{Yuan2}. Then $\widehat{\deg}(\overline{\mathcal C})=0$ and $c=d_B$. In particular, for the indefinite quaternion algebra of discriminant $15$, the height comparison is
\begin{equation*}
    h_\omega(X_K)-h_{\underline\omega_{\mathcal A}}(X_K)=\frac12\log15.
\end{equation*}
This can be read on the maximal-level stack or on a sufficiently small finite cover, using the pull-backs of its two metrized line bundles and dividing intersection numbers by the covering degree; see \cite[Remark 1.2]{Yuan2}. Thus the finite-place metrics in this example are fixed by that model, including at the primes omitted from our base.

On the generic fiber, these two line bundles have their respective advantages in specific applications. The canonical bundle $\omega_{X_K/E_0}$ is more concrete and facilitates calculations, while $\underline{\omega}_\mathcal{A}$ maintains better functoriality in the context of mappings between Shimura varieties. Thus, it is beneficial to understand the explicit relationship between them.

\subsection{Idea of the proof and organization}
Following \cite{Yuan2}, we construct the determinant morphism from the Kodaira--Spencer map using the action of $F^+$. We compute its order of vanishing along each special-fiber component by decomposing the Hodge submodule into column modules, obtaining the coefficients $e_Z$ of $D$. By normality, these local computations determine the image of $\psi$ over the integral model. We then compute $\psi$ explicitly over the complex fibers to compare the Faltings and Petersson metrics.

The organization of this paper is given as follows. In $\S$\ref{sec Sh}, we recall the definition of PEL Shimura varieties and their integral models. We also introduce the classification of PEL datum, as well as the explicit description of Hermitian symmetric domains of different types.

In $\S$\ref{sec KS}, we give a detailed definition of the bundles appearing in the main theorems and the notion of the Kodaira--Spencer map. Next, we construct the determinant morphism from the Kodaira--Spencer map, which is nontrivial for general PEL Shimura varieties. We also check the first statement of Theorems \ref{Intro main1} and \ref{Intro main2} by deformation theory and explicit calculations at the places $v\mid d_B$.

In $\S$\ref{sec metric}, we recall a useful trick from \cite{Yuan2}, which is used to explicitly compute an isomorphism between the tangent space of a complex abelian variety and its first structure sheaf cohomology. Then we give an explicit formula of the Kodaira--Spencer map in the complex setting. Finally, we complete our proof of the main theorems by comparing two metrics using the explicit expression, and derive the height comparison in the proper case.

\paragraph*{Acknowledgement} The author is very grateful to Professor Xinyi Yuan for introducing this problem and for providing many helpful suggestions and comments on this paper as well as on earlier work. He would like to thank his friend Weixiao Lu for helpful advice. He would also like to thank Roy Zhao for many useful communications.

\section{PEL Shimura varieties and integral models} \label{sec Sh}
We recall the PEL data and their moduli interpretation, following \cite{Kot,Lan1,Lan2}. We then specify the integral model used for the local calculations at the nonsplit places of $B$.

\subsection{PEL-type Shimura data}
We first define the PEL datum, then recall its classification and the symmetric domains needed for the metric calculation.

\subsubsection*{PEL datum and Shimura variety}
Write $\mathbb S=\operatorname{Res}_{\CC/\RR}\mathbb G_m$. We use the following convention for a PEL datum.
\begin{definition}\label{PEL datum}
A global PEL datum is a tuple
\begin{equation*}
    (B,*,V,\langle\cdot,\cdot\rangle,h)
\end{equation*}
where
\begin{enumerate}
\item $B$ is a finite-dimensional semisimple $\QQ$-algebra;
\item $*$ is a positive involution on $B$, i.e., $\tr_{B_\RR/\RR}(xx^*)>0$ for all $0\ne x\in B_\RR$;
\item $V$ is a finite faithful left $B$-module;
\item $\langle\cdot,\cdot\rangle$ is a non-degenerate $\QQ$-valued alternating form on $V$ such that $\langle bv,w\rangle=\langle v,b^*w\rangle$ for all $v,w\in V$ and $b\in B$. In particular, the induced involution on $\End(V)$ that sends an endomorphism to its adjoint with respect to $\langle\cdot,\cdot\rangle$ extends $*$ on $B\subset\End (V)$;

Let $G/\QQ$ be the algebraic group determined by the functor
\begin{equation*}
    R\mapsto\{x\in\End_{B\otimes R}(V\otimes R)|xx^*\in R^\times\}.
\end{equation*}
\item $h:\mathbb{S}\rightarrow G_\RR$ is a homomorphism, such that $h(\bar{z})=h(z)^*$ for any $z\in\CC^\times$, the symmetric real-valued bilinear form $\langle v,h(i)w\rangle$ on $V_\RR$ is positive-definite, and the induced Hodge structure on $V_\RR$ is of type $(1,0),(0,1)$.
\end{enumerate}
\end{definition}

We call $(V,\langle\cdot,\cdot\rangle)$ a symplectic $B$-module.

Let $X$ be the $G(\RR)$-conjugacy class of $h$. For a compact open subgroup $K\subset G(\BA_f)$, the associated Shimura variety has complex points
\begin{equation*}
    X_K(\CC)=G(\QQ)\backslash(X\times G(\BA_f))/K.
\end{equation*}
Write $V_\CC=V_1\oplus V_0$, where $h(z)$ acts by $z$ on $V_1$ and by $\bar z$ on $V_0$. The reflex field is
\begin{equation*}
    E_0=\QQ[\{\tr(b| V_1)\}_{b\in B}].
\end{equation*}
We write $X_K$ for the canonical model over $E_0$. Its moduli interpretation is recalled in Section \ref{Moduli interpretation and integral model}.

\subsubsection*{Classification of simple factors}
The positive involution preserves each simple factor of $B$, and the symplectic $B$-module decomposes accordingly; see \cite[1.2.1.11]{Lan1}.

If $B$ is simple, its center is a number field $F$, and the fixed field $F^+$ of $*$ is totally real. Write $B=\M_k(D)$, where $D$ is a division algebra with center $F$. Albert's classification gives the following possibilities.

\begin{proposition}\cite[Prop  1.2.1.13]{Lan1}\label{Albert proposition}
There are exactly four possibilities for $D$:
\begin{enumerate}
\item $D=F=F^+$.
\item $F=F^+$, and $D\otimes_{F,\tau} \RR$ is isomorphic to $\M_2(\RR)$ for every archimedean place $\tau$ of $F$, with the involution $*$ given by conjugating the natural involution $x\mapsto x':=\tr_{D/F}(x)-x$ by some element $a\in D$ such that $a^*=-a$. Note that in this case, $a^2$ is totally negative in $F$.
\item $F=F^+$, and $D\otimes_{F,\tau} \RR$ is isomorphic to the real Hamilton quaternion algebra $\BH$ for every archimedean place $\tau$ of $F$, with the natural involution $*$ given by $x\mapsto x^*:=\tr_{D/F}(x)-x$.
\item $F$ is a totally imaginary quadratic extension of $F^+$, with the complex conjugation $c$, and $D$ satisfies the condition that if $v=v\circ c$ then $\inv_v(D)=0$, and if $v\ne v\circ c$ then $\inv_v(D)+\inv_{v\circ c}(D)=0$.
\end{enumerate}
\end{proposition}

Here $v$ is any place of $F$, $\inv_v(D)\in\QQ/\ZZ$ is the local invariant of $D$, such that $\inv_v(D)=0$ for all but finitely many $v$, and
\begin{equation*}
    \sum_v \inv_v(D)=0.
\end{equation*}
The Albert--Brauer--Hasse--Noether theorem determines $D$ from these invariants. The corresponding classification for matrix algebras is given in \cite[Proposition 1.2.1.14]{Lan1}.

\begin{definition}\cite[Def 1.2.1.15]{Lan1}\label{type of PEL Shimura variety}
Let $B$ be a finite-dimensional semisimple algebra over $\QQ$ with positive involution. We say that a simple factor with center $F$ is of type C (resp. type D) if, for a real embedding $\tau:F\rightarrow\RR$, it becomes $\M_k(\RR)$ (resp. $\M_k(\BH)$), for some integer $k\geq1$, with its positive involution. We say that it is of type A if $F/F^+$ is CM; in this case, for a real embedding $\tau:F^+\rightarrow\RR$, the factor becomes $\M_k(\CC)$ with its positive involution. We say that $B$ involves a factor of the corresponding type.
\end{definition}

\begin{example}
An indefinite quaternion algebra $B/\QQ$ with $V=B$ gives the quaternionic Shimura curves in \cite{Yuan2}. Over a totally real field, a totally indefinite quaternion algebra with $V=B$ gives the twisted Hilbert modular varieties in \cite{Guo1}. Taking $B$ to be a CM field gives unitary Shimura varieties; see \cite{RSZ} for related moduli constructions.
\end{example}

For the explicit computations, we impose the following assumptions.

\begin{assumption}\label{assumption type}
\begin{enumerate}
\item In the decomposition of $B$ into simple factors, no factor of type $D$ appears.
\item Each simple factor of $B$ is a division algebra.
\end{enumerate}
\end{assumption}

\begin{remark}\label{assumption remark}
For Assumption 1, type D factors can fail the Hasse principle; see \cite[A.7.2]{Lan3}. We restrict the computations in this paper to types A and C.

Morita equivalence associates to a left $\M_n(D)$-module $V$ a left $D$-module $W$ with $V\cong W^n$. Transporting the involution and alternating form gives an equivalence of polarized modules, under which the groups of similitudes and the compatible complex structures agree. The associated Shimura varieties are therefore identified. The change in the moduli interpretation is described in Remark \ref{Morita equivalence remark}.
\end{remark}

\subsubsection*{Hermitian symmetric domain}
We recall the matrix realizations of the symmetric domains used in Section \ref{sec metric}; see \cite[Sections 3.1--3.2]{Lan2}.

\begin{definition}
Let $n\ge0$ be an integer. Define the Siegel upper half-space
\begin{equation*}
    \mathcal{H}_n:=\{Z=X+iY\in\Sym_n(\CC), Y>0\}.
\end{equation*}
Here $\Sym_n$ denotes the space of $n\times n$ symmetric matrices, with $X=\Re(Z)$ and $Y=\Img(Z)$ in $\Sym_n(\RR)$, and $Y>0$ means that $Y$ is positive definite. The group $\Sp_{2n}(\RR)$ acts on $\HH_n$ by
\begin{equation*}
    gZ:=(AZ+B)(CZ+D)^{-1},\qquad g=\matrixx{A}{B}{C}{D}.
\end{equation*}
Moreover,
\begin{equation*}
    \HH_n\cong\Sp_{2n}(\RR)/\U_n(\RR),
\end{equation*}
where $\U_n(\RR)$ is the stabilizer of $i1_n$.
\end{definition}

\begin{definition}
Let $a\ge b\ge0$ be any integers. Define the Unitary symmetric domain
\begin{equation*}
    \mathcal{D}_{a,b}:=\left\{U\in\M_{a,b}(\CC), 1_b-U^*U>0\right\}.
\end{equation*}
Here $\M_{a,b}$ denotes the space of $a\times b$ matrices, and $U^*$ is the conjugate transpose of $U$. $\mathcal{D}_{a,b}$ admits a natural left-action of $g=\matrixx{A}{B}{C}{D}\in\U_{a,b}(\RR)$:
\begin{equation*}
    gU:=(AU+B)(CU+D)^{-1},
\end{equation*}
where
\begin{equation*}
    U_{a,b}(\RR)=\left\{g\in\GL_{a+b}(\CC), g^* 1_{a,b}g=1_{a,b}\right\},\quad 1_{a,b}=\matrixx{1_a}{}{}{-1_b}.
\end{equation*}
Moreover,
\begin{equation*}
    \mathcal{D}_{a,b}\cong \U_{a,b}(\RR)/\big(\U_a(\RR)\times \U_b(\RR)\big),
\end{equation*}
where $\U_a(\RR)\times\U_b(\RR)$ is the stabilizer of $0$.
\end{definition}

\begin{remark}
The domain $\mathcal D_{a,b}$ is bounded. When $a=b$, we will use the following unbounded realization for the period calculation.
\end{remark}

\begin{definition}
Let $b\ge 0$ be an integer. Define the Hermitian upper half-space
\begin{equation*}
    \HH_{b,b}=\{Z=X+iY\in\M_b(\CC)\cong\Herm_b(\CC)\otimes_\RR\CC, Y>0\}.
\end{equation*}
Here $\Herm_b$ denotes the space of $b\times b$ Hermitian matrices,
\begin{equation*}
    X:=\frac{1}{2}(Z+Z^*),\ Y:=\frac{1}{2i}(Z-Z^*),\ X,Y\in\Herm_b(\CC).
\end{equation*}
$\HH_{b,b}$ admits a natural left-action of $g=\matrixx{A}{B}{C}{D}\in\U'_{b,b}(\RR)$:
\begin{equation*}
    gZ:=(AZ+B)(CZ+D)^{-1},
\end{equation*}
where
\begin{equation*}
    U'_{b,b}(\RR)=\left\{g\in\GL_{2b}(\CC), g^* J_{b,b}g=J_{b,b}\right\},\quad J_{b,b}=\matrixx{}{1_b}{-1_b}{}.
\end{equation*}
Moreover,
\begin{equation*}
    \HH_{b,b}\cong\mathcal{D}_{b,b}\cong\U'_{b,b}(\RR)/(\U_b(\RR)\times \U_b(\RR)),
\end{equation*}
where the denominator is the stabilizer of $i1_b$.
\end{definition}

For convenience, we denote by
\begin{equation*}
    X=\Reh(Z)=\frac{1}{2}(Z+Z^*),\ Y=\Imh(Z)=\frac{1}{2i}(Z-Z^*),
\end{equation*}
where ``h'' means Hermitian.

These domains describe the connected components of $X$.

\begin{proposition}\label{Classification of Hermitian symmetric domains}
Given a global PEL datum as in Definition \ref{PEL datum}, suppose that $B$ satisfies Assumption \ref{assumption type}. Then each connected component of $X$ is a product of some Siegel upper half-spaces and Unitary symmetric domains.
\end{proposition}

\begin{proof}
Decompose $B$ into simple factors and $V$ into the corresponding orthogonal summands. The compatible complex structures are chosen independently on these summands. For types C and A, their connected components are the Siegel and unitary domains above, respectively; see \cite[Chapter 3]{Lan2}.
\end{proof}

\subsection{Moduli interpretation and integral model}\label{Moduli interpretation and integral model}
We specify the integral PEL datum and level, and construct the integral model at the primes used in our local calculations.

\subsubsection*{Order and PEL-type lattice}
We follow \cite[Sections 1.1--1.2]{Lan1} for the integral PEL datum.

An order in a finite-dimensional $\QQ$-algebra is a subring containing the identity that is also a full $\ZZ$-lattice. It is maximal if it is not properly contained in another order. We use the same terminology over the rings of integers of number fields and local fields.

Now suppose $F$ is a number field and $B/F$ is a central simple algebra of degree $n$, and $\mathcal{O}$ is an $\mathcal{O}_F$-order in $B$. Then the reduced trace pairing
\begin{equation*}
    B\times B\longrightarrow F,\qquad (x,y)\longmapsto\tr_{B/F}(xy)
\end{equation*}
is nondegenerate. The discriminant $\Disc=\Disc_{\mathcal{O}/F}$ is the ideal of $\mathcal{O}_F$ generated by the set of elements
\begin{equation*}
    \{\det(\tr_{B/F}(x_ix_j))_{1\le i\le n^2, 1\le j\le n^2}:x_1,\ldots,x_{n^2}\in\mathcal{O}\}.
\end{equation*}
The local discriminant is given by \cite[Theorem 14.9]{Re}.

\begin{proposition}
Let $\mathcal{O}$ be a maximal order in $B$, and let $v$ be a finite place of $F$ such that $B_v$ is a division algebra over $F_v$. Then
\begin{equation*}
    \Disc_{\mathcal{O}/F,v}=v^{n(n-1)}.
\end{equation*}
\end{proposition}
Maximal orders are conjugate over a nonarchimedean local field; see \cite{Re}.

As in the introduction, put
\begin{equation*}
    d_B=\prod_{v\ \nonsplit\ \mathrm{in}\ B}v^{n_v}.
\end{equation*}
Here $v$ denotes a finite place of $F$, while we say $v$ is nonsplit (or ramified) in $B/F$ if
\begin{equation*}
    B_v:=B\otimes_F F_v\cong\M_{n_v}(D_v)
\end{equation*}
for some division algebra $D_v\ne F_v$ and positive integer $1\le n_v<n$. According to Proposition \ref{Albert proposition}, the nonsplit places in the CM case occur in conjugate pairs with the same $n_v$. Thus $d_B$ is the extension of an ideal of $\mathcal O_{F^+}$, which we also denote by $d_B$. We put $\Nm(d_B)=\Nm_{F^+/\QQ}(d_B)$. For a product of simple factors, these conventions are applied to each factor.

The integral datum is defined as follows.

\begin{definition}\cite[Def 1.2.1.3]{Lan1}\label{PEL type O lattice}
Keep all the notations as in Definition \ref{PEL datum}. A PEL-type $\mathcal O$-lattice is a tuple
\begin{equation*}
    (\mathcal{O},*,L,\langle\cdot,\cdot\rangle,h)
\end{equation*}
where
\begin{enumerate}
\item $\mathcal{O}$ is an order in $B$ mapped to itself under $*$;
\item $L\subset V$ is an $\mathcal{O}$-lattice (of full rank);
\item $\langle\cdot,\cdot\rangle$ is a non-degenerate $\ZZ$-valued alternating form on $L$, which is compatible with the one on $V$.
\end{enumerate}
\end{definition}

We say $L$ is self-dual if the pairing $\langle\cdot,\cdot\rangle$ is perfect. Denote by $L^\vee$ the dual lattice of $L$ under this pairing. Throughout this paper, we make the following assumption on $\mathcal{O}$ and $L$.

\begin{assumption}\label{assumption order}
\begin{enumerate}
\item $\mathcal{O}$ is a maximal order in $B$.
\item $F/\QQ$ is unramified at all places $p$ such that $p|\Nm(d_B)$, where $\Nm$ denotes the numerical norm just defined.
\item $L$ is free as an $\mathcal{O}$-module.
\end{enumerate}
\end{assumption}

\begin{remark}
At a rational prime below a nonsplit place of $B$, the center is unramified and the algebra splits after an unramified extension. The maximal order then gives the hereditary orders used in Section \ref{Image over the integral model}. A maximal order over a Dedekind domain is hereditary, so its lattices are projective; see \cite[Theorem 21.4]{Re}. Freeness of the chosen global lattice $L$ is an additional assumption that specifies the de Rham module in our local calculations.
\end{remark}

\subsubsection*{Moduli interpretation and integral model}

Let $K=\prod_pK_p$ be a compact open subgroup of $\operatorname{Stab}_{G(\BA_f)}(L\otimes_\ZZ\widehat\ZZ)$. We require $K_p=\operatorname{Stab}_{G(\QQ_p)}(L\otimes_\ZZ\ZZ_p)$ at every odd prime $p\mid\Nm(d_B)$. Put
\begin{equation*}
    m=m(K):=\prod_{p:\ K_p\ne\operatorname{Stab}_{G(\QQ_p)}(L\otimes_\ZZ\ZZ_p)}p,\qquad N=2m|\operatorname{disc}(F)|.
\end{equation*}
Here $\operatorname{disc}(F)$ is the absolute discriminant of the center; for a product of simple factors, we take the product of these discriminants. The integral model is defined over $\ZZn$. By Assumption \ref{assumption order}, every odd prime below $d_B$ remains in this base. We take $K$ sufficiently small away from these primes so that the moduli spaces below are schemes.

The level convention is relative to the fixed lattice.
\begin{remark}\label{maximal remark}
By maximal level at $p$ we mean $\operatorname{Stab}_{G(\QQ_p)}(L\otimes\ZZ_p)$. For the quaternionic Shimura curves in \cite{Yuan2}, the corresponding subgroup is $\mathcal O_{B,p}^{\times}$.
\end{remark}

Following \cite[Definition 1.4.1.4]{Lan1}, the generic PEL moduli problem associates to an $E_0$-scheme $S$ the category of tuples $(A,\lambda,i,\eta_K)$, where
\begin{enumerate}
\item $A$ is an abelian scheme of dimension $\frac{1}{2}\dim_\QQ V$ over $S$;
\item $\lambda$ is a fixed polarization of $A$, such that $\mathrm{Ros}_\lambda(i(a))=i(a^*)$, where $\mathrm{Ros}$ is the Rosati involution;
\item $i:\mathcal{O}\rightarrow\End_S(A)$ defines an $\mathcal{O}$-structure of $(A,\lambda)$;
\item the action on $\Lie(A/S)$ satisfies $\det(T-i(b)\mid\Lie(A/S))=\det(T-b\mid V_1)$ for every $b\in\mathcal O$, with the right-hand side viewed in $\mathcal O_S[T]$ through $E_0$; this is the determinant condition of \cite[Definition 1.3.4.1]{Lan1};
\item $\eta_K$ is an integral $K$-level structure of type $(L,\langle\cdot,\cdot\rangle)$ in the sense of \cite[Definition 1.3.7.6]{Lan1}.
\end{enumerate}

At a finite place $v$, recall that the maximal order $\mathcal O_{B_v}$ is Azumaya over $\mathcal O_{F_v}$ precisely when $B_v$ is split. At primes $p\nmid N$ where the order is Azumaya, the same moduli problem with prime-to-$p$ level structure, gives the usual smooth integral PEL model. At the remaining primes, we use the abelian multichain construction below.

\begin{remark}
In type A, the determinant condition specifies the signatures of $V$; see \cite[(2.7)]{RSZ}.
\end{remark}

\begin{remark}\label{Morita equivalence remark}
Under the Morita correspondence in Remark \ref{assumption remark}, replacing $D$ and $W$ by $\M_n(D)$ and $W^n$ replaces the abelian scheme $A$ by $A^n$, with the corresponding action and polarization. Idempotent decomposition gives the inverse construction. The modular curve example appears in \cite[Section 5.1]{Yuan2}.
\end{remark}

We use the principal polarization case throughout the integral calculations.

\begin{assumption}\label{Assumption polarization}
The lattice $L$ is self-dual, and $\lambda$ is a principal polarization.
\end{assumption}

In the remaining of this section, we construct the integral model from the self-dual multichain determined by $L$, then record the local model comparison used in Section \ref{Kodaira Spencer map section}. For $p\nmid N$, put $L_p=L\otimes_\ZZ\ZZ_p$. At $v\mid p$, write $B_v=\M_{n_v}(D_v)$ and $L_v=L\otimes_{\mathcal O_F}\mathcal O_{F_v}$. Choose a uniformizer $\Pi_v$ of $D_v$ and use the lattices $\Pi_v^jL_v$, $j\in\ZZ$, with their normalizer identifications. Since $*$ preserves the radical of the maximal order, or exchanges it with the radical on the conjugate factor, self-duality of $L_p$ makes this multichain self-dual. Its stabilizer is $K_p$, since a $B\otimes\QQ_p$-linear automorphism commutes with $\Pi_v$.

Recall that the abelian multichain moduli problem of \cite[Definitions 6.5, 6.7, and 6.9]{RZ} parametrizes abelian schemes indexed by these lattices, with isogenies of the given degrees, normalizer periodicity, compatible polarizations, the determinant condition and prime-to-$p$ level structure. The member at $L_p$ has the principal polarization of Assumption \ref{Assumption polarization}. Integral prime-to-$p$ level specifies actual abelian schemes within the prime-to-$p$ isogeny classes; see \cite[Proposition 1.4.3.4]{Lan1} and \cite[Proposition 2.2.11]{Lan4}.

To obtain a global moduli scheme, we represent the local lattices by lattices in $V$ agreeing with $L$ away from $p$ and add the finite isogeny data at the finitely many primes where the order is not Azumaya, as in \cite[Choices 2.2.9 and Section 2.2]{Lan4}. The generic fiber of this moduli scheme may contain additional components; see \cite[Remark 1.4.3.12]{Lan1}. We identify $X_K$ with the corresponding open and closed subscheme of the generic fiber, and define $\mathcal X_K$ to be its schematic closure over $\ZZn$. The member at $L$ gives the universal abelian scheme $\mathcal A$.

For the local comparison, write $\mathcal X_{K,p}$ for the base change to $\ZZn\otimes_\ZZ\ZZ_p$ and $\mathcal G_p$ for the similitude group scheme of the polarized multichain. Let $M_p^{\mathrm{loc}}$ be the flat closure of the generic flag variety in the scheme of Hodge filtrations respecting the action, inclusions, normalizer periodicity, duality and determinant condition, as in \cite[Definition 2.3.1 and Remark 2.3.2]{Har}. Then we have the following proposition.

\begin{proposition}\label{local model comparison}
For each $p\nmid N$, $\mathcal G_p$ is the parahoric group scheme with $\mathcal G_p(\ZZ_p)=K_p$, and the de Rham multichain is \'etale locally isomorphic to the fixed polarized multichain. Its torsor of trivializations gives a diagram
\begin{equation*}
    \mathcal X_{K,p}\longleftarrow\widetilde{\mathcal X}_{K,p}\longrightarrow M_p^{\mathrm{loc}},
\end{equation*}
whose first morphism is a $\mathcal G_p$-torsor and whose second morphism is smooth of relative dimension $\dim G$. The second morphism records the Hodge filtration of the universal abelian multichain. In particular, the de Rham module at $L_p$ is locally the fixed module $L_p\otimes_{\ZZ_p}\mathcal O_{\mathcal X_{K,p}}$ with its action and pairing.
\end{proposition}

\begin{proof}
First, we check the parahoric description. The group scheme $\mathcal G_p$ is smooth by \cite[Theorems 3.11 and 3.16]{RZ}; its polarized similitude formulation and connected fibers in types A and C follow from \cite[Definition 21.6.2, Theorem 21.6.4, Corollary 21.6.6, and Remark 21.6.7]{SW}. After an unramified extension splitting the center and the division algebras, Morita equivalence gives a periodic general linear lattice chain in type A, of period $n/n_v$, and a self-dual symplectic chain in type C.

For the nonsplit type C factor, let $R$ be a strictly henselian unramified discrete valuation ring splitting $B_v$. The row lattices of $\mathcal O_{B_v}^r\otimes R$ are $R^{2r}$ and $pR^r\oplus R^r$. The two adjoint matrices computed in Step 1 of Proposition \ref{type C local calculation} give the following dualities for the alternating form on the multiplicity space, denoted by $\#$:
\begin{equation*}
    \begin{aligned}
    (R^{2r})^\#&=R^{2r},&
    (pR^r\oplus R^r)^\#&=p^{-1}(pR^r\oplus R^r)
    &&\text{in the first case},\\
    (R^{2r})^\#&=pR^r\oplus R^r,&
    (pR^r\oplus R^r)^\#&=R^{2r}
    &&\text{in the second case}.
    \end{aligned}
\end{equation*}
Indeed, perfectness of the full pairing gives these identities using the tensor factor $\operatorname{diag}(p,1)/p$ or $\left(\begin{smallmatrix}0&1\\1&0\end{smallmatrix}\right)$, respectively. In the first case, a symplectic basis adapted to $(pR^r\oplus R^r)/pR^{2r}$ gives the chain through vertices $0$ and $r$. In the second, the alternating elementary divisors are $p^{-1}$ and $1$, each in $r/2$ blocks; thus $r$ is even and the chain has vertex $r/2$. The unramified center has unit different. These chain stabilizers are parahoric by \cite[Section 1.1.11]{KP}.

Now we apply the local model diagram to conclude this proposition. By \cite[Proposition 2.2.4 and Remark 2.3.7]{Har}, the polarized de Rham trivializations form a $\mathcal G_p$-torsor, and the induced morphism to the local model is smooth of relative dimension $\dim G$. By normalizer periodicity, the Hodge filtrations at $\Pi_v^jL_v$ are determined by the filtration at $L_v$.

Note that schematic closure commutes with smooth base change, so this diagram restricts to the full flat closures; see \cite[Section 1.2, pp.~140--141]{PRS}. After the unramified extension as above, the lattice-chain models are flat by \cite[Theorems 4.19 and 4.21]{Go1} and \cite[Theorem 2.1]{Go2}. Their flat closures are the parahoric local models by \cite[Proposition 2.3.7]{KP}, and are normal by \cite[Theorem 2.1]{HR}. Hence the full flat closure of the moduli scheme is normal, and the closures of its generic connected components are open and closed. Thus the local model diagram restricts to $\mathcal X_{K,p}$ which is the schematic closure of $X_K$ defined above.
\end{proof}

Finally, we apply this comparison and the known properties of parahoric local models to establish the geometric properties of $\mathcal X_K$ needed in Section \ref{sec KS}.

\begin{proposition}\label{prop normal cm model}
The integral model $\mathcal X_K$ is flat and normal over $\ZZn$, is Cohen--Macaulay, and has geometrically reduced special fibers. It is smooth outside the rational primes below $d_B$. Denote by $\XXK^{\sm}$ its relative smooth locus and put $\XXK^\sing=\XXK\backslash\XXK^\sm$. Then $\XXK^\sing$ has codimension at least two in $\XXK$.
\end{proposition}

\begin{proof}
At a prime $p\nmid N$, the residue characteristic is odd, the center is unramified, and $G$ splits over an unramified extension. The derived groups of the factors of types A and C are simply connected. Thus \cite[Theorem 2.1 and Remark 2.2]{HR} apply to the parahoric local models, giving normality, Cohen--Macaulayness and geometric reducedness of their special fibers. Flatness follows from the construction by flat closure. Proposition \ref{local model comparison} transfers these properties to $\mathcal X_K$. Where $B$ is split at every place above $p$, the chosen self-dual maximal-order lattice gives a hyperspecial local model, which is smooth.

The generic fiber is smooth. Each generic point of a special fiber is also smooth over the base, since the fiber is geometrically reduced. The relative smooth locus therefore contains every codimension-one point of $\mathcal X_K$, proving the last assertion.
\end{proof}

Without the smallness assumption on $K$, the same constructions give a Deligne--Mumford stack. Such a stack has an \'etale cover by schemes, so the properties above and the constructions below descend in the \'etale topology. Thus we may carry out the computations on sufficiently small level schemes.

\section{Kodaira--Spencer map and its image over the integral model}\label{sec KS}
We construct the determinant morphism from the Kodaira--Spencer map and compute its integral image. The extra action of the center gives the determinant comparison, and the local Hodge submodules determine the orders along the special fibers. This extends the method of \cite[Chapter 3]{Guo1}.

Keeping the notation and assumptions of the previous section, we make the following assumption for the explicit computation.

\begin{assumption}
$B$ is a central division algebra over $F$ of degree $n$, possibly $1$. In type C, $F=F^+$ is totally real; in type A, $F$ is a CM field with maximal totally real subfield $F^+$.
\end{assumption}
For convenience, suppose $g=[F^+:\QQ]$, and the symplectic $B$-module $V$ has dimension $2rn$ over $F^+$.

\subsection{Kodaira--Spencer map}\label{Kodaira Spencer map section}
We first define the Hodge bundle, the relative dualizing sheaf and their metrics. Deformation theory describes the Kodaira--Spencer pairing through the algebra action and the polarization. We then use the determinant over the center, as in \cite[Section 3.3]{Guo1}, to construct the morphisms in the main theorems.

\subsubsection*{Arithmetic line bundles and exact sequence}
With the notation for the PEL Shimura variety above, denote the universal abelian scheme by $\pi:\mathcal{A}\longrightarrow\mathcal{X}_K$ with  $\lambda:\mathcal{A}\longrightarrow\mathcal{A}^t$ the universal principal polarization. Also denote by $\epsilon:\mathcal{X}_K\longrightarrow\mathcal{A}$ and $\epsilon^t:\mathcal{X}_K\longrightarrow\mathcal{A}^t$ the identity sections.

Let $\Omega_{\mathcal{A}/\mathcal{O}_{E_0}[\frac{1}{N}]}$, $\Omega_{\mathcal{A}/\mathcal{X}_K}$ and $\Omega_{\mathcal{X}_K/\mathcal{O}_{E_0}[\frac{1}{N}]}$ be the relative differential sheaves, $\omega_{\mathcal{A}/\mathcal{X}_K}$ and $\omega_{\mathcal{X}_K/\mathcal{O}_{E_0}[\frac{1}{N}]}$ be the relative dualizing sheaves, $\mathcal{T}_{\mathcal{A}/\mathcal{X}_K}$ and $\mathcal{T}_{\mathcal{X}^\sm_K/\mathcal{O}_{E_0}[\frac{1}{N}]}$ be the relative tangent sheaves. Furthermore, we have the Lie algebra
\begin{equation*}
    \Lie(\mathcal{A}):=\epsilon^*\mathcal{T}_{\mathcal{A}/\mathcal{X}_K}\cong \pi_*\mathcal{T}_{\mathcal{A}/\mathcal{X}_K},
\end{equation*}
and the Hodge bundles
\begin{equation*}
    \underline{\Omega}_{\mathcal{A}}:=\epsilon^*\Omega_{\mathcal{A}/\mathcal{X}_K}\cong \pi_*\Omega_{\mathcal{A}/\mathcal{X}_K},     \underline{\omega}_{\mathcal{A}}:=\epsilon^*\omega_{\mathcal{A}/\mathcal{X}_K}\cong \pi_*\omega_{\mathcal{A}/\mathcal{X}_K}.
\end{equation*}
These definitions clearly imply canonical isomorphisms
\begin{equation*}
    \underline{\Omega}_{\mathcal{A}}\cong \Lie(\mathcal{A})^{\vee},     \underline{\omega}_{\mathcal{A}}\cong \det\underline{\Omega}_{\mathcal{A}}.
\end{equation*}
In fact, the Hodge bundle $\underline{\omega}_{\mathcal{A}}$ on $\mathcal{X}_K$ has an important metric called the Faltings metric $\lVert\cdot\rVert_{\Fal}$ as follows. Suppose $l=rng$. For any point $x\in\mathcal{X}_K(\mathbb{C})$, let $\alpha\in\underline{\omega}_{\mathcal{A}}(x)\cong\Gamma(\mathcal{A}_x,\omega_{\mathcal{A}_x/\mathbb{C}})$, where $\mathcal{A}_x$ is the fiber of $\mathcal{A}$ above $x$ which is a complex abelian variety of dimension $l$, then $\alpha$ is a holomorphic $l$-form. The Faltings metric is defined by
\begin{equation*}
    \lVert\alpha\rVert_{\Fal}^2:=\frac{1}{(2\pi)^l}\left|\int_{\mathcal{A}_x(\mathbb{C})}\alpha\wedge\bar{\alpha}\right|.
\end{equation*}
See \cite{Yuan1} about this canonical hermitian metric for example.

By Proposition \ref{prop normal cm model} and \cite[Lemma 5.10]{HR}, the dualizing sheaf $\WU$ is reflexive of rank one. Its restriction to $\mathcal X_K^{\sm}$ is the determinant of the relative cotangent bundle, and
\begin{equation*}
    \WU\cong\left(\bigwedge^{\dim X_K}\OU\right)^{\vee\vee}.
\end{equation*}
On the complex fibers, the canonical bundle has the Petersson metric $\lVert\cdot\rVert_{\Pet}$. We give its definition according to the two types of PEL data.

When $B$ is of type C, or equivalently $F=F^+$ under Assumption \ref{assumption type}, $\dim\XK=gr(r+1)/2$ and we can choose $(Z_1,\cdots,Z_g)$ to be the coordinates of $\HH_r^g$, where each $Z_i$ is a symmetric complex matrix with positive definite imaginary part $Y_i$. Denote by $\de\tau_i:=\wedge^{j\ge k} \de Z_{i,jk}$, where $Z_{i,jk}$ is the $(j,k)$-element in matrix $Z_i$ for $1\le i\le g$. Then clearly $\de\tau:=\wedge_{i=1}^g\de\tau_i$ is a nowhere-vanishing top differential form on the symmetric domain, and the Petersson metric is defined by
\begin{equation*}
    \lVert\de\tau\rVert_{\Pet}:=2^{\frac{gr(r+1)}{2}}\prod_{i=1}^g\det(Y_i)^{\frac{r+1}{2}}.
\end{equation*}
Note that when $n=g=1$, this defines the classical Petersson metric of dualizing sheaf over the Siegel modular variety.

When $B$ is of type A, or equivalently $F/F^+$ a CM field,
\begin{equation*}
    \dim X_K=\sum_{i=1}^g p_i\cdot q_i,\ p_i+q_i=r,
\end{equation*}
where $i$ corresponds to the pairs of conjugate complex places of $F$, $(p_i,q_i)$ is the signature of the Hermitian space defined by $V$ at the corresponding place. The Petersson metric can also be defined for arbitrary signatures; see Remark \ref{Petersson metric remark} for an example. Here we consider the case when $p_i=q_i=\frac{r}{2}$. In this case, suppose $(Z_1,\cdots,Z_g)$ are the coordinates of $\prod_{i=1}^g\mathcal{H}_{\frac{r}{2},\frac{r}{2}}$. Similarly, denote by $\de\tau_i:=\wedge^{1\le j,k\le \frac{r}{2}} \de Z_{i,jk}$. Then $\de\tau:=\wedge_{i=1}^g\de\tau_i$ is a nowhere-vanishing top differential form on the symmetric domain, and the Petersson metric is defined by
\begin{equation*}
    \lVert\de\tau\rVert_{\Pet}:=2^\frac{gr^2}{4}\prod_{i=1}^g\det(Y_i)^{\frac{r}{2}},
\end{equation*}
where $Y_i=\Imh(Z_i)$ is the positive Hermitian imaginary part of $Z_i$. Note that when $g=1,r=2$, this definition agrees with the previous one.

\begin{remark}\label{Petersson metric remark}
In \cite{Guo2}, the author considered the unitary Shimura variety, and the Hermitian symmetric domain $D$ is $\mathcal{D}_{n,1}$ in our notation. In fact, under the standard coordinate in \cite[Chap 5.3]{Guo2} with $d_{W,\iota}=1$, the definition of $D$ there using projective lines agrees perfectly with the current definition of $\mathcal{D}_{n,1}$ using matrices. Then the metric $h_{L_D}$ defined on the tautological bundle in \cite[Chap 1.1]{Guo2} there can be realized as a Petersson metric, which equals $1-U^*U$ for the corresponding matrix $U\in\M_{n,1}(\CC)$.
\end{remark}

Now we are ready to introduce the Kodaira--Spencer map as follows. For the following construction we restrict the sheaves to $\mathcal X_K^{\sm}$. Consider the exact sequence
\begin{equation*}
    0\longrightarrow\pi^*\Omega_{\mathcal{X}_K/\mathcal{O}_{E_0}[\frac{1}{N}]}\longrightarrow\Omega_{\mathcal{A}/\mathcal{O}_{E_0}[\frac{1}{N}]}\longrightarrow\Omega_{\mathcal{A}/\mathcal{X}_K}\longrightarrow 0.
\end{equation*}
Applying the derived functors of $\pi_*$ gives a connecting morphism
\begin{equation*}
    \phi_0:\pi_*\Omega_{\mathcal{A}/\mathcal{X}_K}\longrightarrow R^1\pi_*(\pi^*\Omega_{\mathcal{X}_K/\mathcal{O}_{E_0}[\frac{1}{N}]}).
\end{equation*}
We call this the Kodaira--Spencer map. Note that there are canonical isomorphisms
\begin{equation*}
    \begin{aligned}
    R^1\pi_*(\pi^*\OU)&\cong R^1\pi_*\mathcal O_{\mathcal A}\otimes\OU\\
    &\cong\Lie(\mathcal A^t)\otimes\OU\cong\underline\Omega_{\mathcal A^t}^{\vee}\otimes\OU,
    \end{aligned}
\end{equation*}
where $\mathcal{A}^t$ denotes the dual universal abelian scheme. Here we give some explanation. The first isomorphism holds by projection formula. The second isomorphism is just $R^1\pi_*\mathcal{O}_{\mathcal{A}}\cong\pi_*^t \mathcal{T}_{\mathcal{A}^t/\mathcal{X}_K}$, where $\pi^t:\mathcal{A}^t\longrightarrow\mathcal{X}_K$ is induced by $\pi$. This follows from deformation theory, see \cite[Remark 9.4(c)]{Mil} for example. The third one is just our definition. Then the Kodaira--Spencer map can be also written as
\begin{equation}\label{phi_1}
    \phi_1:\underline{\Omega}_{\mathcal{A}}\longrightarrow\underline{\Omega}_{\mathcal{A}^t}^\vee\otimes\Omega_{\mathcal{X}_K/\mathcal{O}_{E_0}[\frac{1}{N}]}.
\end{equation}
\subsubsection*{Isomorphism from deformation theory}
Now we introduce the isomorphism from deformation theory. Note that \eqref{phi_1} also gives
\begin{equation}\label{phi_2}
    \phi_2:\underline\Omega_{\mathcal A}\otimes\underline\Omega_{\mathcal A^t}\longrightarrow\OU.
\end{equation}
The principal polarization identifies $\underline\Omega_{\mathcal A^t}$ with $\OA$. Under this identification the pairing is symmetric, and compatibility with the Rosati involution gives the relations in the following theorem.

\begin{theorem}\label{deformation}
On the generic fiber, there is an isomorphism induced by \eqref{phi_2}
\begin{equation}\label{phi_3}
    \phi_3:(\OA\otimes\OA)/\mathcal R\xrightarrow{\;\sim\;}\Omega_{X_K/E_0},
\end{equation}
where $\mathcal R$ is locally generated by
\begin{equation*}
    (i(\beta)^*u)\otimes v-u\otimes(i(\beta^*)^*v),\qquad u\otimes v-v\otimes u,
\end{equation*}
for $u,v\in\OA$ and $\beta\in\mathcal O_B$. The same assertion holds on the open subset of $\mathcal X_K^{\sm}$ where $\mathcal O_B$ is Azumaya over its center, with $\Omega_{X_K/E_0}$ replaced by the relative cotangent bundle there.
\end{theorem}

\begin{proof}
The deformation description in \cite[Proposition 2.3.5.2]{Lan1} identifies the pairing between $\underline\Omega_{\mathcal A}$ and $\underline\Omega_{\mathcal A^t}$ with the variation of the Hodge filtration. The relations are those imposed by the algebra action and the polarization. On the stated integral open subset the center is \'etale and the order is Azumaya, so Morita equivalence gives the same description over the base. The principal polarization is an isomorphism, and the identity $\lambda\circ i(\beta)=i(\beta^*)^t\circ\lambda$ changes the relations for the dual action into the two relations above. This proves the assertion.
\end{proof}

For the integral computation in later subsections, we use Proposition \ref{local model comparison} together with Grothendieck--Messing theory. At a geometric point of a special fiber in $\mathcal X_K^{\sm}$, deformations of first order are described by lifts of the Hodge filtrations satisfying the local model conditions. Note that the prime-to-$p$ level structure lifts uniquely; while for places $v$ dividing $p$, normalizer periodicity relates the Hodge filtrations of the multichain $\Pi_v^jL_v$ to that at $L_v$. Compatibility with the lattice inclusions follows from $\mathcal O_{B_v}$-stability, since $\Pi_v\in\mathcal O_{B_v}$.

Next, we use the Gauss--Manin connection to identify these first-order variations of the Hodge filtrations with the Kodaira--Spencer morphism \eqref{phi_1}; see \cite[Remark 2.1.7.4]{Lan1}. Under the crystalline comparison, a compatible  graph deformation $f\mapsto f+\varepsilon u(f)$ of first order with $\varepsilon^2=0$, gives $f\mapsto u(f)$ in the Hodge quotient. This is $\phi_1$ evaluated on the corresponding tangent vector. Then we evaluate $\phi_1$ against $\underline\Omega_{\mathcal A^t}$ to obtain $\phi_2$, since the principal polarization identifies $\underline\Omega_{\mathcal A^t}$ with $\underline\Omega_{\mathcal A}$. Thus the evaluation pairings computed in Propositions \ref{explicit basis of R} and \ref{type C local calculation} are the polarized Kodaira--Spencer pairing.

\subsubsection*{Morphism between line bundles}
In general, it is nontrivial to obtain a canonical morphism between powers of $\WA$ and $\WU$ from the Kodaira--Spencer map \eqref{phi_1}. Our construction has three steps, as in \cite[Chapter 3.3]{Guo1}. First we compare determinants over $F^+$ with ordinary determinants. Next we construct the morphism after splitting the center when $n=1$. Finally we apply Morita equivalence and check descent for general $n$. We first work on the generic fiber; the integral image is computed in the next subsection.

\textbf{Step 1.} We compare the two determinants. The embedding $\mathcal O_{F^+}\hookrightarrow\mathcal O_B$ gives $\OA$ an $\mathcal O_{F^+}$-action. Since $*=1$ on $F^+$, the relations defining $\mathcal R$ give the same action on $(\OA\otimes\OA)/\mathcal R$, and hence on the cotangent bundle in \eqref{phi_3}.

Following \cite[Chapter 3.3]{Guo1}, write $\otf$ for tensor product over $\mathcal O_{F^+}\otimes_\ZZ\mathcal O_{\mathcal X_K^{\sm}}$, and $\df$ for determinant over this algebra. We continue to use $\otimes$ and $\det$ for the ordinary tensor product and determinant over $\mathcal O_{\mathcal X_K^{\sm}}$. On the generic fiber the same notation means tensor product and determinant over $F^+\otimes_\QQ\mathcal O_{X_K}$.

\begin{remark}
In type A, the center $F$ is stable under $*$, but $*$ acts on it by complex conjugation. Consequently the relation $i(\beta)^*u\otimes v=u\otimes i(\beta^*)^*v$ pairs the two opposite center factors. The symmetric quotient is naturally a module over $F^+$, which is the reason for taking $\df$ over $F^+$ in both types.
\end{remark}

The comparison of determinants requires the following factor in the determinant of the coefficient algebra.
\begin{lemma}\label{Key lemma}
Let $\mathcal F$ be locally free of rank $k$ over $\mathcal O_{F^+}\otimes_\ZZ\mathcal O_{\mathcal X_K^{\sm}}$. There is a canonical isomorphism
\begin{equation*}
    \det(\mathcal F)\cong\det(\df(\mathcal F))\otimes\det(\mathcal O_{F^+}\otimes_\ZZ\mathcal O_{\mathcal X_K^{\sm}})^{\otimes(k-1)}.
\end{equation*}
Equivalently, if $\Nm$ denotes the norm of an invertible module for this finite \'etale algebra, then
\begin{equation*}
    \det(\mathcal F)\cong\Nm(\df(\mathcal F))\otimes\det(\mathcal O_{F^+}\otimes_\ZZ\mathcal O_{\mathcal X_K^{\sm}})^{\otimes k}.
\end{equation*}
Moreover, if $\mathcal F$ has rank one over this algebra, a morphism $\mathcal M\to\mathcal N\otf\mathcal F$ between locally free modules of rank $k$ induces
\begin{equation*}
    \df(\mathcal M)\longrightarrow\df(\mathcal N)\otf\mathcal F^{\otf k}.
\end{equation*}
\end{lemma}

\begin{proof}
Choose a basis of $\mathcal F$ over the coefficient algebra, locally on $\mathcal X_K^{\sm}$. The ordinary determinant of a change of basis is the norm of its determinant over that algebra. This gives the second formula. The first follows from $\Nm(\mathcal P)=\det(\mathcal P)\otimes\det(\mathcal O_{F^+}\otimes\mathcal O_{\mathcal X_K^{\sm}})^{-1}$ for an invertible module $\mathcal P$. The last assertion is the usual determinant formula for tensoring by an invertible module.
\end{proof}

Thus we may perform the construction separately over the factors of $F^+$ after an \'etale splitting, using $\otf$ and $\df$. Taking norms and applying Lemma \ref{Key lemma} then gives the ordinary determinant, together with its constant factor. This replaces the equality of the two determinants used in \cite[Lemma 3.2]{Guo1}.

\textbf{Step 2.} We construct the morphism when $n=1$ over one factor of the split algebra $F^+$. In type C, the relations give
\begin{equation*}
    (\OA\otimes\OA)/\mathcal R\cong\Sym^2(\OA),\qquad \det(\Sym^2(\OA))\cong\WA^{\otimes(r+1)}.
\end{equation*}
Taking determinants in \eqref{phi_3} therefore gives the canonical isomorphism
\begin{equation*}
    \WA^{\otimes(r+1)}\xrightarrow{\;\sim\;}\omega_{X_K/E_0}
\end{equation*}
over this factor. Here, and until we take norms in Step 3, the determinants are taken over the chosen factor.

In type A, split the quadratic center and order its two idempotents. The two Hodge summands have ranks $p$ and $q$, where $p+q=r$. The following lemma explains why we require $p=q=r/2$ in the main theorem of type A.
\begin{lemma}\label{algebraic lemma}
Let $W=W_1\oplus W_2$ be a direct sum of free modules of ranks $p,q$ over a splitting field, with $p+q=r>0$, on which the two idempotents of the quadratic center act. Let $R\subset W\otimes W$ be generated, for $u,v\in W$ and $\beta$ in the split quadratic center, by
\begin{equation*}
    i(\beta)u\otimes v-u\otimes i(\beta^*)v,\qquad u\otimes v-v\otimes u.
\end{equation*}
Then $(W\otimes W)/R\cong W_1\otimes W_2$, and
\begin{equation*}
    \det((W\otimes W)/R)\cong(\det W_1)^{\otimes q}\otimes(\det W_2)^{\otimes p}.
\end{equation*}
This is naturally a positive tensor power of $\det W$, compatibly with independent changes of basis in $W_1$ and $W_2$, precisely when $p=q=r/2$. Its exponent is then $r/2$.
\end{lemma}

\begin{proof}
Applying the relation for either center idempotent kills $W_1\otimes W_1$ and $W_2\otimes W_2$. The symmetry relation identifies $W_2\otimes W_1$ with $W_1\otimes W_2$. The determinant formula for a tensor product gives the asserted identity. If one summand has rank zero, the quotient determinant is trivial, whereas a positive power of $\det W$ has a nontrivial determinant character. Otherwise, under independent changes of basis, the two determinant characters have exponents $q$ and $p$. They are equal to the same positive exponent precisely when $p=q=r/2$.
\end{proof}

The lemma applies to $\OA$ and $\mathcal R$. For a fixed ordering of the center factors and balanced signature, taking determinants in \eqref{phi_3} gives
\begin{equation*}
    \WA^{\otimes r/2}\xrightarrow{\;\sim\;}\omega_{X_K/E_0}.
\end{equation*}
We check the effect of exchanging the two center factors together with Morita equivalence in the next step.

\textbf{Step 3.} We pass to general $n$ and descend the construction. After an \'etale splitting, $B$ is a matrix algebra over each center factor. The matrix units identify a module with $n$ copies of the module obtained by applying $e_{11}$. This is the Morita equivalence used in Remark \ref{Morita equivalence remark}; in these matrix coordinates, the Hodge determinant is the $n$th power of the determinant of the module obtained by applying $e_{11}$.

In type C, write the involution on $\M_n$ as $b\mapsto Q^{-1}b^tQ$, where $Q$ is an invertible symmetric matrix. The relations in \eqref{phi_3} identify the quotient with $\Sym^2(e_{11}\OA)$, by contracting the two matrix indices with $Q$. Since $e_{11}\OA$ has rank $r$, the determinant identities are
\begin{equation*}
    \det(\OA)\cong\det(\mathcal O^n)^{\otimes r}\otimes\det(e_{11}\OA)^{\otimes n},\qquad \det(\Sym^2(e_{11}\OA))\cong\det(e_{11}\OA)^{\otimes(r+1)}.
\end{equation*}
The determinant pairing for $Q$ identifies $\det(\mathcal O^n)^{\otimes2}$ with $\mathcal O$. Taking its $r(r+1)/2$-th power gives
\begin{equation*}
    (\df\OA)^{\otf(r+1)}\xrightarrow{\;\sim\;}(\df\Omega_{X_K/E_0})^{\otf n}.
\end{equation*}
This construction descends from the splitting of $B$. Indeed, changing the matrix basis by $g$ replaces $Q$ by $g^tQg$; the factor $(\det g)^{r(r+1)}$ in the Hodge determinant is exactly the factor in $(\det Q)^{r(r+1)/2}$. Replacing $Q$ by $cQ$ multiplies this last factor by $c^{nr(r+1)/2}$ and divides the determinant of the induced cotangent identification, taken to the $n$th power, by the same scalar. Every change of a matrix-algebra splitting is given locally by such a change of basis, determined up to scalar. Thus the formula is independent of the splitting, including when $B$ has nontrivial Brauer class.

In type A, denote the two factors of the split center by $e_+,e_-$. The involution induces a perfect pairing between the standard modules of the two matrix factors. The quotient in \eqref{phi_3} is the tensor product of the two modules obtained by Morita equivalence, each of rank $r/2$, with their matrix indices contracted by this pairing. Its determinant pairs the two factors $\det(\mathcal O^n)$, so the determinant formula in Lemma \ref{algebraic lemma} gives
\begin{equation*}
    (\df\OA)^{\otf r/2}\xrightarrow{\;\sim\;}(\df\Omega_{X_K/E_0})^{\otf n}
\end{equation*}
for this ordering. The determinant of the pairing occurs to the $r^2/4$-th power, which cancels the determinants of changes of the two matrix bases. Rescaling the pairing by $c$ multiplies this factor by $c^{nr^2/4}$ and divides the $n$th power of the cotangent determinant by the same scalar. Hence the construction is independent of the choice of identification with a matrix algebra.

Exchanging $e_+$ and $e_-$ multiplies the exterior product of the two Hodge summands by $(-1)^{n^2r^2/4}$. On the quotient it transposes the tensor product of two modules of rank $r/2$, with determinant $(-1)^{r(r-2)/8}$. Consequently the two formulas differ by
\begin{equation*}
    (-1)^{n^2r^3/8-nr(r-2)/8}=(-1)^{nr(r+2)/8}.
\end{equation*}
The generator $e_+\wedge e_-$ of $\det_{F^+}(F)$ also changes sign. Multiplying by its $nr(r+2)/8$-th power gives an isomorphism independent of the ordering,
\begin{equation*}
    (\df\OA)^{\otf r/2}\xrightarrow{\;\sim\;}(\df\Omega_{X_K/E_0})^{\otf n}\otf\det_{F^+}(F)^{\otf nr(r+2)/8}.
\end{equation*}
Finally, we take norms from $F^+$. In type C the $F^+$-ranks of $\OA$ and $\Omega_{X_K/E_0}$ are $nr$ and $r(r+1)/2$. In type A they are $nr$ and $r^2/4$. Applying Lemma \ref{Key lemma} gives the following constant line bundle over $\ZZ[1/N]$:
\begin{equation}\label{constant determinant bundle}
    \mathcal C=\begin{cases}
    \det_{\ZZ[1/N]}(\mathcal O_{F^+}[1/N])^{\otimes nr(r+1)/2},&\text{in type C},\\[2mm]
    \begin{aligned}
    &\det_{\ZZ[1/N]}(\mathcal O_{F^+}[1/N])^{\otimes nr^2/4}\\
    &\quad\otimes\Nm_{F^+/\QQ}\!\left(\det_{\mathcal O_{F^+}[1/N]}(\mathcal O_F[1/N])\right)^{\otimes nr(r+2)/8},
    \end{aligned}&\text{in type A}.
    \end{cases}
\end{equation}
We use the same notation for its pull-backs to $\mathcal X_K$ and $X_K$. In particular, the generic determinant maps just constructed are
\begin{equation*}
    \begin{aligned}
    \WA^{\otimes(r+1)}\otimes\mathcal C^{-1}&\xrightarrow{\;\sim\;}\omega_{X_K/E_0}^{\otimes n}&&\text{in type C},\\
    \WA^{\otimes r/2}\otimes\mathcal C^{-1}&\xrightarrow{\;\sim\;}\omega_{X_K/E_0}^{\otimes n}&&\text{in type A}.
    \end{aligned}
\end{equation*}
The trace pairings for $F^+/\QQ$ and $F/F^+$ on $\mathcal O_{F^+}[1/N]$ and $\mathcal O_F[1/N]$ are perfect. Equip their determinant line bundles with the metrics whose squared norms on real vectors are the absolute values of the determinant trace pairings, and equip their norms with the product metric. This gives $\mathcal C$ a Hermitian metric and a canonical isometric trivialization
\begin{equation}\label{constant square trivialization}
    \mathcal C^{\otimes2}\xrightarrow{\;\sim\;}\mathcal O.
\end{equation}
In ordered idempotent bases after splitting the centers, these determinant generators have norm one. When $n$ is even, the exponents in \eqref{constant determinant bundle} are even, so the trace pairings already trivialize $\mathcal C$ itself. Also note that at places where $\mathcal O_B$ is Azumaya over its center, the determinant construction gives an isomorphism of line bundles over $\mathcal X_K^{\sm}$.

The construction above gives the generic isomorphisms in Theorems \ref{Intro main1} and \ref{Intro main2}. Section \ref{Image over the integral model} computes their orders at codimension-one points and proves the asserted integral images; Section \ref{final comparison} proves the isometry. When $n$ is even, the trivialization of $\mathcal C$ gives the untwisted morphisms with the original exponents. For every $n$, the tensor square and \eqref{constant square trivialization} give untwisted morphisms with twice these exponents and divisor $2D$.

\begin{remark}
The principal polarization is used both to identify $\underline\Omega_{\mathcal A^t}$ with $\OA$ in Theorem \ref{deformation} and to make the alternating pairing on the integral de Rham module perfect. The finite-place calculations and the metric comparison below use this same fixed polarization.
\end{remark}

\subsection{Image over the integral model}\label{Image over the integral model}
We prove the integral image assertions in Theorems \ref{Intro main1} and \ref{Intro main2} by computing $e_Z$, i.e., the coefficient of $Z$ in \eqref{vertical divisor D}. Proposition \ref{prop normal cm model} places every codimension-one point in the relative smooth locus. It therefore suffices to prove $\operatorname{ord}_Z(\psi)=e_Z$ for each special-fiber component $Z$, then extend by normality. Propositions \ref{explicit basis of R} and \ref{type C local calculation} give this equality in types A and C, respectively.

Fix $Z$, with generic point $z$. After an unramified extension separating the center embeddings and splitting the local division algebras, write $\mathcal O$ for the discrete valuation ring over $\mathcal O_{\XXK,z}$ and $\pi$ for its uniformizer. All sums over $\sigma$ below run through the embeddings $F^+\hookrightarrow\operatorname{Frac}(\mathcal O)$. In each summand, the local degree and Hodge multiplicities belong to that embedding; in type A we choose one of the two conjugate factors. The local reflex field is unramified, since inertia fixes the Hodge cocharacter class. Thus smoothness at $z$ makes a uniformizer of the center a uniformizer of $\mathcal O$, and the trace determinant factor $\mathcal C$ has order zero.

As in \cite[Section 3.2]{Yuan2} and \cite[Section 3.4]{Guo1}, we describe the maximal order and compute the Kodaira--Spencer pairing on its column modules. The left action on de Rham cohomology is $b\cdot x=i(b^*)^*x$; the principal polarization identifies it with the left action on homology. Replacing $\beta$ by $\beta^*$ leaves the relations defining $\mathcal R$ unchanged.

\subsubsection*{Cyclic algebra}
Let $v$ be a non-archimedean place of $F$. As above, write
\begin{equation*}
    B_v=\M_{n_v}(D_v),\qquad [D_v:F_v]=(n/n_v)^2.
\end{equation*}
The matrix factor will be treated by integral Morita equivalence. We first give the explicit calculation for the division algebra $D_v$; the local degree is $n/n_v$ throughout this calculation.

\begin{definition}
Choose an unramified extension $E_v/F_v$ of degree $n/n_v$, a generator $\tau$ of its Galois group, and a uniformizer $\pi$ of $F_v$ such that $D_v$ has the cyclic presentation
\begin{equation*}
    D_v=(E_v/F_v,\tau,\pi)=\bigoplus_{i=0}^{n/n_v-1}E_v\cdot u^i,\qquad u^{n/n_v}=\pi,\qquad ux=\tau(x)u\quad(x\in E_v).
\end{equation*}
Such a presentation exists for every central division algebra over $F_v$; see \cite[Thm.~14.5]{Re}.
\end{definition}

The field $E_v$ splits $D_v$. The following lemma computes the maximal order after this base change and generalizes the first assertion of \cite[Lem.~3.2]{Yuan2}.

\begin{lemma}\label{base change of maximal order}
There is an $\mathcal O_{E_v}$-linear ring isomorphism
\begin{equation*}
    \mathcal O_{D_v}\otimes_{\mathcal O_{F_v}}\mathcal O_{E_v}\simeq\left\{(b_{ij})\in\M_{n/n_v}(\mathcal O_{E_v}):b_{ij}\in\pi\mathcal O_{E_v}\text{ for }i>j\right\}.
\end{equation*}
Consequently, $\mathcal O_{B_v}$ becomes the matrix algebra of size $n_v$ over this order.
\end{lemma}

\begin{proof}
The maximal order is $\bigoplus_{i=0}^{n/n_v-1}\mathcal O_{E_v}u^i$ by \cite[Thm.~14.6]{Re}. Under the splitting, its generators act by
\begin{equation*}
    x\longmapsto\begin{pmatrix}x&0&\cdots&0\\0&\tau(x)&\cdots&0\\\vdots&\vdots&\ddots&\vdots\\0&0&\cdots&\tau^{n/n_v-1}(x)\end{pmatrix},\qquad u\longmapsto\begin{pmatrix}0&1&0&\cdots&0\\0&0&1&\cdots&0\\\vdots&\vdots&\vdots&\ddots&\vdots\\0&0&0&\cdots&1\\\pi&0&0&\cdots&0\end{pmatrix}.
\end{equation*}
The first formula, after extension of scalars, gives all diagonal matrices, since $E_v/F_v$ is unramified. The diagonal matrix units and the matrix of $u$ generate exactly the order in the statement. When $n/n_v=1$, the order is $\mathcal O_{F_v}$ and $u=\pi$.
\end{proof}

After extending scalars to $\mathcal O$, denote this hereditary order by
\begin{equation*}
    \Lambda_v=\left\{(b_{ij})\in\M_{n/n_v}(\mathcal O):b_{ij}\in\pi\mathcal O\text{ for }i>j\right\}.
\end{equation*}
Let $e_i$ be its diagonal matrix units. The column modules
\begin{equation*}
    P_i=\Lambda_ve_i=\mathcal O^i\oplus\pi\mathcal O^{n/n_v-i}\subset\mathcal O^{n/n_v},\qquad 1\leq i\leq n/n_v,
\end{equation*}
describe all the integral Hodge submodules needed below. Their multiplicities record the integral module at the chosen component.

\begin{lemma}\label{explicit basis lemma}
Every $\Lambda_v$-module which is finite and free over $\mathcal O$ is a direct sum of the modules $P_i$, with uniquely determined multiplicities. In particular, every such module is projective over $\Lambda_v$.
\end{lemma}

\begin{proof}
The diagonal idempotents and the action of $u$ turn the module into a chain of lattices. Let $M=\bigoplus_i e_iM$ be a finite $\mathcal O$-free $\Lambda_v$-module. The successive maps induced by $u$ are injective, since their composition around the cycle is multiplication by $\pi$. Identifying $e_iM$ with its image $L_i$ in $e_1M=L_1$ gives
\begin{equation*}
    \pi L_1\subset L_{n/n_v}\subset\cdots\subset L_2\subset L_1.
\end{equation*}
The subspaces $L_i/\pi L_1$ form a flag in $L_1/\pi L_1$. We choose a basis adapted to this flag and lift it to an $\mathcal O$-basis of $L_1$. Each corresponding chain of rank one is $\mathcal O$ through a unique position $i$ and $\pi\mathcal O$ thereafter, and hence gives $P_i$. The dimensions of the successive quotients determine its multiplicity. Finally, $P_i=\Lambda_ve_i$ is projective. The case $n/n_v=1$ is the classification of free $\mathcal O$-modules.
\end{proof}

\subsubsection*{Explicit expression of $\mathcal R$}
We compute the quotient defined by the algebra action and the symmetry of the polarized Kodaira--Spencer pairing. In type A, the involution exchanges the two center factors after the center is split. The principal polarization pairs their de Rham modules, so the calculation uses the Hodge submodule in one factor and the dual of its quotient in the other factor.

Fix one embedding of $F^+$ and order the two factors of the split quadratic center. After integral Morita equivalence for $\M_{n_v}(\Lambda_v)$, write $M_E$ for the Hodge submodule in the first factor and $Q_E$ for its quotient. The Hodge filtration gives the following exact sequence of $\Lambda_v$-modules
\begin{equation*}
    0\longrightarrow M_E\longrightarrow\Lambda_v^{rn_v/n}\longrightarrow Q_E\longrightarrow0,\qquad \operatorname{rank}_{\mathcal O}M_E=\operatorname{rank}_{\mathcal O}Q_E=\frac{rn}{2n_v}.
\end{equation*}
For convenience we call it the Hodge sequence. Indeed, Morita equivalence sends the free module $\M_{n_v}(\Lambda_v)^{r/n}$ to $\Lambda_v^{rn_v/n}$. The Hodge submodule in the second center factor is $Q_E^\vee=\operatorname{Hom}_{\mathcal O}(Q_E,\mathcal O)$, with its right $\Lambda_v$-action. All these identifications use the given principal polarization.

Both modules in this Hodge sequence are free over $\mathcal O$. Lemma \ref{explicit basis lemma} makes $Q_E$ projective over $\Lambda_v$, so the sequence splits. There are therefore unique integers $m_i$ with
\begin{equation}\label{eq type a hodge multiplicities}
    M_E\simeq\bigoplus_{i=1}^{n/n_v}P_i^{m_i},\qquad Q_E\simeq\bigoplus_{i=1}^{n/n_v}P_i^{rn_v/n-m_i},\qquad 0\leq m_i\leq\frac{rn_v}{n},\quad\sum_{i=1}^{n/n_v}m_i=\frac r2.
\end{equation}
\begin{proposition}\label{explicit basis of R}
For the decomposition \eqref{eq type a hodge multiplicities}, the polarized Kodaira--Spencer quotient and its determinant have the following descriptions.
\begin{enumerate}
\item The relative tangent module and cotangent module are
\begin{equation*}
    \operatorname{Hom}_{\Lambda_v}(M_E,Q_E),\qquad C_E=\operatorname{Hom}_{\Lambda_v}(M_E,Q_E)^\vee\simeq Q_E^\vee\otimes_{\Lambda_v}M_E.
\end{equation*}
They are free of rank $r^2/4$ over $\mathcal O$. The Kodaira--Spencer pairing is the evaluation map $\ell\otimes f\mapsto(\delta\mapsto\ell(\delta f))$. For a copy of $P_i$ in $M_E$ and a copy of $P_j$ in $Q_E$, the corresponding cotangent coordinate is $\mathcal O$ if $i\leq j$, and $\pi^{-1}\mathcal O$ if $i>j$.
\item The coefficient $e_Z$ in \eqref{vertical divisor D} is the order of $\psi$ along $Z$, and is given by
\begin{equation}\label{eq type a local order}
    e_Z=\operatorname{ord}_Z(\psi)
    =\sum_\sigma\left(\frac{(n-2n_v)r^2}{8}+\frac n2\sum_{i=1}^{n/n_v}m_i^2\right).
\end{equation}
For each $\sigma$, the summand is the order of the determinant morphism on that factor, with $m_i$ as in \eqref{eq type a hodge multiplicities}. It is independent of the chosen conjugate factor and is zero at a split place, where $n_v=n$.
\end{enumerate}
\end{proposition}

\begin{proof}
The proof is divided into three steps. We first identify the quotient by computing evaluation on the column modules, then compare its determinant lattice with the Hodge determinant, and finally check this result is independent of the choices made after the unramified extension.

\textbf{Step 1.} We use the splitting of the Hodge sequence above. On the Grassmannian open set where projection to $M_E$ is invertible, a submodule is the graph of a map $\delta:M_E\to Q_E$. It is stable under $\Lambda_v$ precisely when $\delta$ is $\Lambda_v$-linear. The graph is then isomorphic to $M_E$ and satisfies the same determinant condition. Its annihilator gives the Hodge submodule in the second center factor. Normalizer periodicity determines the remaining Hodge submodules in the multichain from these two, as explained in Section \ref{Kodaira Spencer map section}.

Right multiplication gives
\begin{equation*}
    \operatorname{Hom}_{\Lambda_v}(P_i,P_j)=e_i\Lambda_ve_j=\begin{cases}\mathcal O,&i\leq j,\\\pi\mathcal O,&i>j.\end{cases}
\end{equation*}
These identities commute with base change because $P_i$ is projective. The factor $\pi\mathcal O$ is a free module with basis $\pi$, and its base change is taken with this basis. Hence the graph open set is an affine space, smooth of relative dimension $r^2/4$. Its generic fiber is dense, so it belongs to the flat local model. Dualizing this Hom formula gives the cotangent coordinates and the evaluation isomorphism in the first statement.

The relations that define $\mathcal R$ imply that only pairings between the two conjugate center factors remain. The symmetry relation identifies the two orders of the remaining factors, and the algebra relations are exactly those defining $Q_E^\vee\otimes_{\Lambda_v}M_E$. Thus the evaluation map gives the required polarized quotient. Take the factor for the chosen embedding of $F^+$ in the Grothendieck--Messing comparison of Section \ref{Kodaira Spencer map section}. The resulting tangent map is $v\mapsto(f\mapsto\nabla_v(f)\bmod M_E)$, and its reduction modulo $\pi$ is an isomorphism onto $\operatorname{Hom}_{\Lambda_v}(M_E,Q_E)\otimes\mathcal O/\pi$. Both modules are free of rank $r^2/4$ over $\mathcal O$, so Nakayama's lemma identifies their lattices. Thus the pairing computed above agrees with the Kodaira--Spencer pairing.

\textbf{Step 2.} Now we compare the generic determinant morphism with the integral lattices. On the generic fiber, Morita equivalence writes $M_E$ and $Q_E$ as the tensor products of the standard representation of dimension $n/n_v$ with two vector spaces of dimension $r/2$. The determinants of the standard representation cancel between $M_E$ and $Q_E^\vee$. The usual tensor-product determinant formula therefore gives
\begin{equation*}
    (\det M_E\otimes\det Q_E^\vee)^{\otimes r/2}\simeq(\det C_E)^{\otimes n/n_v}.
\end{equation*}
This is the generic determinant morphism constructed in Section \ref{Kodaira Spencer map section}, after the chosen splitting of the center.

The determinant lattices of $P_i$ and $P_i^\vee$ in the standard coordinate determinant spaces are $\pi^{n/n_v-i}\mathcal O$ and $\pi^{-(n/n_v-i)}\mathcal O$. Consequently, the exponent of $\det M_E\otimes\det Q_E^\vee$ is
\begin{equation*}
    \sum_{i=1}^{n/n_v}\left(2m_i-\frac{rn_v}{n}\right)\left(\frac n{n_v}-i\right),
\end{equation*}
whereas the Hom formula in Step 1 gives exponent $-\sum_{i>j}m_i(rn_v/n-m_j)$ for $\det C_E$. The order of the preceding morphism is thus
\begin{equation*}
    \begin{aligned}
    &\frac r2\sum_{i=1}^{n/n_v}\left(2m_i-\frac{rn_v}{n}\right)\left(\frac n{n_v}-i\right)+\frac n{n_v}\sum_{i>j}m_i\left(\frac{rn_v}{n}-m_j\right)\\
    &\hspace{15mm}=\left(\frac n{n_v}-1\right)\frac{r^2}{4}-\frac n{n_v}\sum_{i>j}m_im_j
    =\left(\frac n{n_v}-2\right)\frac{r^2}{8}+\frac n{2n_v}\sum_i m_i^2.
    \end{aligned}
\end{equation*}
Here we used $\sum_i m_i=r/2$. Integral Morita equivalence for $\M_{n_v}(\Lambda_v)$ replaces the combined Hodge determinant by its $n_v$-th tensor power and preserves the tangent module. The determinant morphism for the original module is therefore the $n_v$-th tensor power of the morphism just computed, giving the summand in \eqref{eq type a local order}. This order is an integer by the first expression. It is nonnegative when $n/n_v\geq2$. If $n_v=n$, then $m_1=r/2$, and it is zero.

\textbf{Step 3.} We descend the computation from the unramified extension. Galois permutes the diagonal idempotents cyclically, so the sum of the squares of the multiplicities is unchanged. To check the exchange of the two center factors, the actual dual column modules are
\begin{equation*}
    P_i^\vee=\pi^{-1}e_{i+1}\Lambda_v\quad(i<n/n_v),\qquad P_{n/n_v}^\vee=e_1\Lambda_v.
\end{equation*}
Indeed, their row coordinates are $\mathcal O^i\oplus\pi^{-1}\mathcal O^{n/n_v-i}$. Thus choosing the other center factor replaces the multiplicities by $rn_v/n-m_i$, up to permutation. Their squares have the same sum, since
\begin{equation*}
    \sum_{i=1}^{n/n_v}\left(\frac{rn_v}{n}-m_i\right)^2=\frac n{n_v}\left(\frac{rn_v}{n}\right)^2-2\frac{rn_v}{n}\frac r2+\sum_i m_i^2=\sum_i m_i^2.
\end{equation*}
The valuation is unchanged under an unramified extension. Adding the orders over $\sigma$ therefore gives $\operatorname{ord}_Z(\psi)$ and the coefficient $e_Z$ defined in \eqref{vertical divisor D}, proving \eqref{eq type a local order}.
\end{proof}

We now compute $e_Z$ for type C by the same evaluation and determinant comparison as in Proposition \ref{explicit basis of R}. The additional discussion concerns the involution and symmetric pairings between the two column modules.

\begin{proposition}\label{type C local calculation}
Suppose $B$ is of type C. At a quaternionic division factor, $n=2$ and
\begin{equation*}
    \Lambda_v=\begin{pmatrix}\mathcal O&\mathcal O\\\pi\mathcal O&\mathcal O\end{pmatrix},\qquad
    M_E\simeq P_1^m\oplus P_2^{r-m}\subset\Lambda_v^r,
\end{equation*}
with the column modules $P_1,P_2$ defined above. After a strictly henselian unramified extension, a change of basis preserving $\Lambda_v$ gives $\beta^*=S^{-1}\beta^tS$, where
\begin{equation*}
    S=\begin{pmatrix}\pi&0\\0&1\end{pmatrix}
    \qquad\text{or}\qquad
    S=\begin{pmatrix}0&1\\1&0\end{pmatrix}.
\end{equation*}
These cases have odd and even discriminant valuation, respectively; in the second case $r$ is even and $m=r/2$. The coefficient $e_Z$ in \eqref{vertical divisor D} is
\begin{equation}\label{eq type c local order}
    e_Z=\operatorname{ord}_Z(\psi)=\sum_\sigma
    \begin{cases}
    0,&B_v\text{ is split},\\[1mm]
    \binom{m+1}{2}+\binom{r-m+1}{2},&B_v\text{ is division, odd valuation},\\[1mm]
    r^2/4,&B_v\text{ is division, even valuation}.
    \end{cases}
\end{equation}
For each $\sigma$, the parity refers to the discriminant of the orthogonal involution on that factor, and $m$ is its Hodge multiplicity. If $C_E$ denotes its cotangent module, the corresponding nonzero summand is the order of $(\det M_E)^{\otimes(r+1)}\to(\det C_E)^{\otimes2}$. In the odd case this order depends only on $\{m,r-m\}$ and equals $r(r+1)/2$ for $m=0$ or $m=r$.
\end{proposition}

\begin{proof}
At split places, the symmetric-square determinant computation in Section \ref{Kodaira Spencer map section} is an integral isomorphism. Now we treat a division factor by first identifying the polarized modules and then applying the determinant comparison of Proposition \ref{explicit basis of R}.

\textbf{Step 1.} We determine the involution and the dual Hodge submodule. The normalizer of $\Lambda_v$ permutes the two homothety classes $P_1,P_2$ and is generated by scalars, $\Lambda_v^\times$ and the matrix of $u$ in Lemma \ref{base change of maximal order}. Comparing $S$ with $\operatorname{diag}(\pi,1)$ gives, up to a scalar,
\begin{equation*}
    \begin{pmatrix}\pi a&\pi b\\\pi b&d\end{pmatrix},\quad a,d\in\mathcal O^\times,
    \qquad\text{or}\qquad
    \begin{pmatrix}\pi b&a\\a&c\end{pmatrix},\quad a\in\mathcal O^\times,
\end{equation*}
Here $b,c\in\mathcal O$. In the first case, an elementary basis change preserving $\Lambda_v$ diagonalizes the matrix, and square roots of units give $\operatorname{diag}(\pi,1)$. In the second, Hensel's lemma applied to $b+2at+\pi ct^2=0$ gives the isotropic vector $(1,\pi t)$. It pairs with the second standard vector by a unit; scaling that vector and subtracting a multiple of the first gives the second matrix in the statement. These basis changes preserve $\Lambda_v$. The resulting symmetric matrices have determinant valuations $1$ and $0$, respectively.

In type C, the $\mathcal O$-dual carries the left action $(b\ell)(x)=\ell(b^*x)$. Identifying duals through $x^tSy$ yields
\begin{equation*}
    \begin{aligned}
    P_1^\vee&=\pi^{-1}P_2,&P_2^\vee&=\pi^{-1}P_1&&\text{in the first case},\\
    P_1^\vee&=\pi^{-1}P_1,&P_2^\vee&=P_2&&\text{in the second case}.
    \end{aligned}
\end{equation*}
The polarization identifies $\Lambda_v^r/M_E$ with $M_E^\vee$. This quotient is projective by Lemma \ref{explicit basis lemma}, so the Hodge sequence splits. The multiplicities in $M_E\oplus M_E^\vee\simeq P_1^r\oplus P_2^r$ force $2m=r$ in the second case and impose no further restriction in the first.

\textbf{Step 2.} We compute the symmetric evaluation pairing and its determinant. A $\Lambda_v$-stable complement $Q_E$ can be made Lagrangian by replacing it with the graph of the map $t:Q_E\to M_E$ determined by $\langle t(q),q'\rangle=-\frac12\langle q,q'\rangle$. Since $2$ is invertible, averaging with the transpose makes symmetric forms a direct summand of $\operatorname{Hom}_{\Lambda_v}(M_E,M_E^\vee)$, compatibly with base change. These Lagrangian graph deformations form an affine space of relative dimension $r(r+1)/2$. The determinant condition, normalizer periodicity and identification of the evaluation pairing with the actual Kodaira--Spencer map follow as in Step 1 of Proposition \ref{explicit basis of R}.

The cotangent coordinate for two column modules is the ideal generated by their pairings $x^tSy$: the integral symmetric forms have coefficients in the dual ideal. For $(P_1,P_1)$, $(P_1,P_2)$ and $(P_2,P_2)$, these ideals are, respectively,
\begin{equation*}
    (\pi\mathcal O,\pi\mathcal O,\mathcal O)
    \qquad\text{or}\qquad
    (\pi\mathcal O,\mathcal O,\mathcal O).
\end{equation*}
Thus, with $v_\pi$ the normalized valuation,
\begin{equation*}
    v_\pi(\det C_E)=
    \begin{cases}
    m(m+1)/2+m(r-m),&\text{in the first case},\\
    m(m+1)/2,&\text{in the second case}.
    \end{cases}
\end{equation*}
As in Step 2 of Proposition \ref{explicit basis of R}, we compare the generic determinant map with these lattices. The symmetric-square determinant identity makes contraction by $S$ give $(\det S)^{r(r+1)/2}$, and $v_\pi(\det M_E)=m$, so its order is
\begin{equation*}
    (r+1)m+\frac{r(r+1)}2v_\pi(\det S)-2v_\pi(\det C_E).
\end{equation*}
This equals $r(r+1)/2-m(r-m)$ in the first case and $m(r-m)=r^2/4$ in the second. Rescaling $S$ cancels between the determinant factor and $2v_\pi(\det C_E)$, and exchanging $P_1,P_2$ preserves the result. Adding these orders over $\sigma$ gives $\operatorname{ord}_Z(\psi)=e_Z$ and proves \eqref{eq type c local order}.
\end{proof}

\subsubsection*{Two explicit examples}
We give the two quaternionic curve calculations and compare them with the preceding formulas.

\begin{example}[Quaternionic Shimura curve]
The setup is the same as \cite[Chap.~1]{Yuan2}: $B$ is an indefinite quaternion algebra over $\QQ$, $V=B$, and the positive involution is of the first kind, so that $r=g=1$ and $n=2$. At a division place the involution has the first form in Proposition \ref{type C local calculation}. For $M_E=P_2$, choose its standard basis $x,y$. Then
\begin{equation*}
    u x=\pi y,\qquad u y=x,\qquad a x=ax,\qquad a y=\tau(a)y\quad(a\in\mathcal O_{E_v}).
\end{equation*}
The symmetric contraction with $S=\operatorname{diag}(\pi,1)$ has values $x\otimes x\mapsto\pi$, $y\otimes y\mapsto1$, and $x\otimes y,y\otimes x\mapsto0$. Its kernel is generated by
\begin{equation*}
    x\otimes y,\qquad y\otimes x,\qquad x\otimes x-\pi y\otimes y.
\end{equation*}
For $M_E=P_1$, the cotangent coordinate ideal is $\pi\mathcal O$ and $v_\pi(\det M_E)=1$. In both cases the determinant morphism $(\det M_E)^{\otimes2}\to C_E^{\otimes2}$ has order $1$. This is the local order in \cite[Sec.~3.2]{Yuan2}.
\end{example}

\begin{example}[Quaternionic unitary Shimura curve]
Let $B$ be a quaternion division algebra over an imaginary quadratic field $F$, let $V=B$, and choose a positive involution of the second kind and signature $(1,1)$. Then $g=1$ and $r=n=2$. At a division place, $n_v=1$, and \eqref{eq type a hodge multiplicities} gives $(m_1,m_2)=(1,0)$ or $(0,1)$. Since $g=1$, formula \eqref{eq type a local order} gives $e_Z=1$ in either case, agreeing with the quaternionic curve calculation. This includes the unitary construction used in \cite[Sec.~1.1.1]{Zhang} whenever its PEL lattice and polarization satisfy our assumptions.

The same calculation also records the component dependence in higher dimension. If $n=2$, $n_v=1$, and $r=4$, the decompositions $M_E=P_1^2$, $Q_E=P_2^2$ give order $4$, whereas $M_E=Q_E=\Lambda_v$ gives order $2$. In the first case the determinant exponents of $M_E$, $Q_E^\vee$, and $C_E$ are $2,0,0$. In the second case the basis $e_{11},e_{12},\pi e_{21},e_{22}$ of $\Lambda_v$ gives exponents $1,-1,-1$. The determinant morphism in these cases is $(\det M_E\otimes\det Q_E^\vee)^{\otimes2}\to(\det C_E)^{\otimes2}$, giving the stated orders directly.
\end{example}

Propositions \ref{explicit basis of R} and \ref{type C local calculation} prove that the coefficient $e_Z$ in \eqref{vertical divisor D} equals $\operatorname{ord}_Z(\psi)$. Their formulas are nonnegative and vanish where $B$ is split, so $D$ is effective and supported over finitely many primes. We now use this same divisor to extend $\psi$ over $\XXK$.

We finish the proof of the integral parts in Theorems \ref{Intro main1} and \ref{Intro main2}. On a normal scheme, a reflexive sheaf of rank one is the intersection of its stalks at codimension-one points inside its generic fiber. The Hodge determinant with its constant factor is invertible, and $(\WU^{\otimes n})^{\vee\vee}$ is reflexive. The generic morphism is an isomorphism at every horizontal prime divisor, and the computations above give its nonnegative order $e_Z$ at each vertical prime divisor. Therefore, it extends to $(\WU^{\otimes n})^{\vee\vee}$. This extension is unique and injective, since it is an isomorphism on the generic fiber and both sheaves are torsion free.

Twisting $(\WU^{\otimes n})^{\vee\vee}$ by $-D$ and taking the double dual removes exactly these orders. Then the induced morphism between reflexive sheaves has order zero at every prime divisor and is therefore an isomorphism. This gives
\begin{equation*}
    \begin{cases}\WA^{\otimes(r+1)}\otimes\mathcal C^{-1},&\text{in type C},\\\WA^{\otimes r/2}\otimes\mathcal C^{-1},&\text{in type A},\end{cases}\quad\xrightarrow{\ \sim\ }\bigl(\WU^{\otimes n}\otimes\mathcal O_{\XXK}(-D)\bigr)^{\vee\vee}.
\end{equation*}
In particular, its right-hand side is invertible, and the effective divisor $D$ gives its inclusion in $(\WU^{\otimes n})^{\vee\vee}$.

Over a local reflex ring with uniformizer $\pi$, the special-fiber components of $\mathcal X_K^{\sm}$ are disjoint. Removing those whose coefficient differs from a fixed $e_Z$ gives an open subset containing the generic fiber on which $D=e_Z\operatorname{div}(\pi)$ and $\operatorname{Im}(\psi)=\pi^{e_Z}\WU^{\otimes n}$. These open subsets cover $\mathcal X_K^{\sm}$, while \eqref{vertical divisor D} also describes the image across the intersections of components in the full model.

Also note that taking tensor squares and using the trace-pairing trivialization of $\mathcal C^{\otimes2}$ gives the corresponding identities with exponents $2(r+1)$ and $2n$ in type C, and $r$ and $2n$ in type A, and with divisor $2D$. When $n$ is even, the canonical trivialization of $\mathcal C$ gives the untwisted identities with the original exponents and divisor $D$.

\section{Comparison of the metrics}\label{sec metric}
In this section we compare the Faltings and Petersson metrics under the determinant morphism constructed in Section \ref{Kodaira Spencer map section}. We first describe the map between tangent spaces induced by the polarization, then compute the connecting morphism in period coordinates, and finally use the volume of the period lattice to compute the norm of its image. For the explicit type A calculation, we make the following assumption.

\begin{assumption}
\begin{enumerate}
\item $g=1$, or equivalently $F^+=\QQ$.
\item The PEL datum is of type A.
\end{enumerate}
\end{assumption}

For a general totally real field $F^+$, the connecting morphism is computed at each embedding of $F^+$ into $\RR$. The volume, however, is computed using the entire integral homology lattice in the product of the real factors. We give this calculation in Section \ref{final comparison}, together with the type C calculation. This also corrects the volume calculation in \cite[Chap 4]{Guo1}: the trace pairing need not be self-dual over $\ZZ$, so its determinant must be included when computing the volume. The same determinant factor occurs in the explicit morphism and in the Faltings norm, giving the required equality of norms.

Recall the classification of irreducible Hermitian symmetric domains of PEL Shimura varieties in Proposition \ref{Classification of Hermitian symmetric domains}. We consider the exact sequence
\begin{equation*}
    0\longrightarrow\pi^*\Omega_{X/\mathbb{C}}\longrightarrow\Omega_{\mathcal{A}/\CC}\longrightarrow\Omega_{\mathcal{A}/X}\longrightarrow0.
\end{equation*}
Here $X$ is the Siegel upper half-space $\HH_r$ or the Hermitian upper half-space $\Hr$, and $\pi:\mathcal{A}\rightarrow X$ denotes the pull-back of the universal abelian scheme to the symmetric domain. Composing the connecting morphism with the inverse of the differential of the principal polarization gives
\begin{equation}\label{com1}
    \phi:\underline{\Omega}_{\mathcal{A}/X}\longrightarrow\Lie(\mathcal{A}/X)\otimes\Omega_{X/\mathbb{C}}.
\end{equation}
We compute this morphism and then take the determinants prescribed in Section \ref{Kodaira Spencer map section}.

\subsection{Explicit map between complex tangent spaces}\label{Cech}
We recall the cohomological description of the differential of a polarization from \cite[Section 4.1]{Yuan2}, which will be used to compute the connecting morphism.

In order to compute $\phi$ explicitly, it is important to understand the following isomorphism
\begin{equation*}
    V\longrightarrow\Lie(A)\longrightarrow\Lie(A^t)\longrightarrow H^1(A,\mathcal{O}_A).
\end{equation*}
Here $V$ is a complex vector space, $A=V/\Lambda$ is an abelian variety, with a principal polarization $\lambda$ induced by a Riemann form $E:\Lambda\times\Lambda\longrightarrow\mathbb{Z}$ for $\Lambda$ a lattice in $V$. The first map is the differential of the quotient map $V\to A$, the second is induced by the polarization, while the third is the canonical isomorphism by deformation of line bundles; see \cite[Remark 9.4(c)]{Mil}.

To make this isomorphism explicit, we need to describe the map to $H^1(A,\mathcal{O}_A)$. We will use \v{C}ech cohomology to introduce an explicit canonical map
\begin{equation*}
    h:\Hom_{\mathbb{Z}}(\Lambda,\mathbb{C})\longrightarrow H^1(A,\mathcal{O}_A),
\end{equation*}
which is a special case for the explicit homomorphism
\begin{equation*}
    \delta:H^1(\Lambda,\mathcal{F}(V))\longrightarrow H^1(A,\mathcal{F}_A).
\end{equation*}
Here $\mathcal{F}_A$ is a sheaf on $A$, and $\mathcal{F}$ is its pull-back to $V$, with the action of $\Lambda$ given by translations.

We first need a suitable open cover of $A$ coming from a cover of $V$. Take a set $\{U_t\}_{t\in I}$ that is a family of open subsets $U_t\subset V$, satisfying the following condition:
\begin{enumerate}[(1)]
\item each composition $U_t\longrightarrow V\longrightarrow A$ is injective;
\item $\bigcup_{t\in I}U_t\longrightarrow A$ is surjective;
\item for any $t,t'\in I$, the difference of two sets $U_{t'}-U_t$, also as a subset of $V$, contains at most one point of $\Lambda$. Denote by $c_{t,t'}$ this point if it exists.
\end{enumerate}
Denote by $\bar{U}_t$ the image of $U_t\longrightarrow A$, which then forms a cover of $A$. Obviously, such a cover exists just by taking each $U_t$ small enough. We call such a cover an admissible cover of $A$.

Now, we define $\delta$ explicitly. For any cross-homomorphism $\alpha:\Lambda\longrightarrow\mathcal{F}(V)$, define a \v{C}ech cocycle $\delta(\alpha)$ such that each component $\delta(\alpha)_{t,t'}\in\mathcal{F}(\bar{U}_{t,t'})$ on $\bar{U}_{t,t'}:=\bar{U}_t\cap\bar{U}_{t'}$ is given by the image of $\alpha(c_{t,t'})$ under the composition $\mathcal{F}(V)\longrightarrow\mathcal{F}(U_t)\longrightarrow\mathcal{F}(\bar{U}_t)\longrightarrow\mathcal{F}(\bar{U}_{t,t'})$. The cross-homomorphism identity and the uniqueness of $c_{t,t'}$ give the cocycle condition. Coboundaries map to coboundaries, so this defines $\delta$.

Taking $\mathcal F=\mathcal O$, define $h$ to be the composition
\begin{equation}\label{composition}
    \Hom_{\mathbb{Z}}(\Lambda,\mathbb{C})\longrightarrow H^1(\Lambda,\mathbb{C})\longrightarrow H^1(\Lambda,\mathcal{O}(V))\longrightarrow H^1(A,\mathcal{O}_A).
\end{equation}
Here, the first isomorphism follows from $\Lambda$ acting on $\mathbb{C}$ trivially, and the second map is induced by the natural map $\mathbb{C}\longrightarrow\mathcal{O}(V)$. In particular, $h$ kills the restriction to $\Lambda$ of every complex linear function on $V$: its constant cocycle is the coboundary of that function.

Finally, we have the following lemma.
\begin{lemma}\label{app}
The composition
\begin{equation*}
    V\longrightarrow\Lie(A)\longrightarrow\Lie(A^t)\longrightarrow H^1(A,\mathcal{O}_A)
\end{equation*}
is given by
\begin{equation*}
    z\mapsto2\pi i h(E(z,\cdot)),
\end{equation*}
where $E$ is extended to the underlying real vector space of $V$, and $E(z,\cdot)$ is viewed as an element of $\Hom_{\mathbb{Z}}(\Lambda,\mathbb{C})$.
\end{lemma}
The lemma is proved in \cite[Lemma 4.1]{Yuan2} using the Appell--Humbert theorem. Although $E$ is real bilinear, the composition in the lemma is complex linear. This distinction will be used when passing from the period coordinates to cohomology.

\subsection{Explicit connecting morphism}
We compute \eqref{com1} in period coordinates, using the Riemann form of the fixed principal polarization and Lemma \ref{app}.

\subsubsection*{Universal abelian scheme}
Recall the moduli interpretation of $\mathcal{X}_K$ in Section \ref{Moduli interpretation and integral model}. On a fixed connected component, choose real coordinates for the polarized $B\otimes_{\QQ}\RR$-module and denote by $\Gamma\subset\M_{n,r}(\CC)$ the image of its integral homology lattice. We choose these coordinates so that the involution is $x^*=\mu\bar{x}^t\mu^{-1}$, where $\mu$ is a positive real diagonal matrix, and the compatible alternating form is given by \eqref{general Riemann form} below. Such coordinates exist by diagonalizing the positive Hermitian form defining the involution and putting the remaining skew-Hermitian form of signature $(r/2,r/2)$ into its standard form. All these changes of coordinates are also applied to the integral homology lattice. In particular, the principal polarization gives a perfect pairing on $\Gamma$.

With these coordinates, the pull-back of the universal abelian scheme is
\begin{equation*}
    \mathcal{A}=\Gamma\backslash(\Hr\times\mathbb{C}^{nr}).
\end{equation*}
To describe the action, let $\beta=(\beta_{ij})\in\Gamma$, $Z\in\Hr$, and $\zeta=(z_{ij})\in\M_{n,r}(\CC)$. The two matrices
\begin{equation*}
    \matrices{Z}{I},\qquad\matrices{Z^t}{I}
\end{equation*}
belong to $\M_{r,r/2}(\CC)$, where $I=I_{r/2}$. The action is
\begin{equation*}
    \beta\cdot(Z,\zeta)=\left(Z,\zeta+\left(\beta\matrices{Z}{I},\bar\beta\matrices{Z^t}{I}\right)\right).
\end{equation*}
Thus, for each $Z\in\Hr$, we have the uniformization
\begin{equation}\label{canonical uniformization}
    \mathcal{A}_Z(\CC)=\CC^{nr}/\Lambda_Z,\qquad\Lambda_Z=\left\{\left(\beta\matrices{Z}{I},\bar\beta\matrices{Z^t}{I}\right):\beta\in\Gamma\right\}.
\end{equation}
Writing $\beta=(a,b)$ with $a,b\in\M_{n,r/2}(\CC)$, the period map is $(a,b)\mapsto(aZ+b,\bar a Z^t+\bar b)$. Since $Y=(Z-Z^*)/(2i)$ is positive definite, this is an invertible real linear map. Indeed, conjugating the second period coordinate gives the complex linear map $(a,b)\mapsto(aZ+b,aZ^*+b)$, whose determinant is nonzero. Hence $\Lambda_Z$ is a full lattice in $\CC^{nr}$. The first $r/2$ columns transform through the chosen complex embedding of the center, and the last $r/2$ through its conjugate, as required by the signature. The action on $\Hr$ is trivial.

\begin{remark}
The same construction gives the universal abelian scheme over the bounded realization $\mathcal{D}_{p,q}$ for arbitrary signatures. In this case $Z\in\M_{p,q}(\CC)$, and the two period matrices are
\begin{equation*}
    \matrices{Z}{I_q}\in\M_{r,q}(\CC),\qquad\matrices{I_p}{Z^t}\in\M_{r,p}(\CC).
\end{equation*}
The coordinates and the transported homology lattice are chosen for the corresponding diagonal Hermitian form of signature $(p,q)$. The condition $I_q-Z^*Z>0$ makes the period map invertible over $\RR$, so its image is again a full lattice.
\end{remark}

We now write the Riemann form in these coordinates. First consider the standard alternating form on $\M_{n,r}(\CC)$,
\begin{equation*}
    E(\beta,\beta')=2\Re\tr\left(\beta\matrixx{0}{-I}{I}{0}\bar\beta'{}^t\right)=2\Re\tr(b\bar a'{}^t-a\bar b'{}^t).
\end{equation*}
Under \eqref{canonical uniformization}, this form is alternating and of type $(1,1)$, and $E(z,iz)>0$ for $z\ne0$. The Riemann form of our principal polarization is
\begin{equation}\label{general Riemann form}
    E_B(\beta,\beta')=2\Re\tr\left(\mu^{-1}\beta\matrixx{0}{-I}{I}{0}\bar\beta'{}^t\right).
\end{equation}
It is integral and self-dual on $\Gamma$ by the choice of the transported lattice. For $B=F$, the expression $2\Re$ is the real extension of $\tr_{F/\QQ}$. Its Gram determinant in the ordinary real coordinates is $2^{2nr}$ when $\mu=I$, whereas the Gram determinant of $E_B$ is
\begin{equation*}
    2^{2nr}(\det\mu)^{-2r}.
\end{equation*}
Consequently, self-duality relates $\det\mu$ to the volume of the actual lattice $\Gamma$, as we compute in Section \ref{final comparison}. This relation gives the volume factor needed in the metric comparison.

\subsubsection*{Explicit connecting morphism}
Keeping the above notation, the differentials $\{\de z_{ij}\}_{1\le i\le n,\,1\le j\le r}$ form a basis of $\underline{\Omega}_{\mathcal{A}/\Hr}$, and their pull-backs form a basis of $\Omega_{\mathcal{A}/\Hr}$. Thus
\begin{equation*}
    \underline{\omega}_{\mathcal{A}/\Hr}=\mathcal{O}_{\Hr}\bigwedge_{1\le i\le n,\,1\le j\le r}\de z_{ij},\qquad\Lie(\mathcal{A}/\Hr)=\bigoplus_{1\le i\le n,\,1\le j\le r}\mathcal{O}_{\Hr}\frac{\partial}{\partial z_{ij}}.
\end{equation*}
We also have
\begin{equation*}
    \Omega_{\Hr/\CC}=\bigoplus_{1\le j,k\le r/2}\mathcal{O}_{\Hr}\de Z_{jk},\qquad\omega_{\Hr/\CC}=\mathcal{O}_{\Hr}\bigwedge_{1\le j,k\le r/2}\de Z_{jk}=\mathcal{O}_{\Hr}\de\tau.
\end{equation*}
The notation $\de\tau$ agrees with the definition of the Petersson metric. Let $t$ be the generator of the constant line bundle $\mathcal C$ obtained from the ordered idempotent bases of the split centers. Its norm for the trace determinant metric is $1$.

We now state the explicit formula, obtained by the same \v{C}ech calculation as in \cite[Sec 4.2]{Yuan2}.
\begin{theorem}\label{Explicit archimedean}
The connecting morphism \eqref{com1} over the Hermitian upper half-space is given by
\begin{equation*}
    \phi(\de z_{ij})=\begin{cases}
    \displaystyle\sum_{k=1}^{r/2}\sum_{l=1}^n\frac{\mu_{li}}{2\pi i}\frac{\partial}{\partial z_{l(k+r/2)}}\otimes\de Z_{kj},&1\le j\le r/2,\\[6pt]
    \displaystyle\sum_{k=1}^{r/2}\sum_{l=1}^n\frac{\mu_{li}}{2\pi i}\frac{\partial}{\partial z_{lk}}\otimes\de Z_{(j-r/2)k},&r/2+1\le j\le r.
    \end{cases}
\end{equation*}
Consequently, the determinant morphism
\begin{equation*}
    \psi:\underline{\omega}_{\mathcal A}^{\otimes r/2}\otimes\mathcal C^{-1}\longrightarrow\omega_{\Hr/\CC}^{\otimes n}
\end{equation*}
satisfies, with the exterior-product ordering fixed in its construction,
\begin{equation*}
    \psi\left(\left(\bigwedge_{1\le i\le n,\,1\le j\le r}\de z_{ij}\right)^{\otimes r/2}\otimes t^{-1}\right)=\left(\frac{\det\mu}{(2\pi i)^n}\right)^{r^2/4}(\de\tau)^{\otimes n}.
\end{equation*}
\end{theorem}

\begin{proof}
We differentiate the period translations, use the polarization to identify the resulting cohomology classes with tangent vectors, and then take determinants. Fix $Z^0\in\Hr$.

\textbf{Step 1.} We compute the connecting morphism in cohomology. We claim that
\begin{equation*}
    \phi_0:\underline{\Omega}_{\mathcal{A}/\Hr}\longrightarrow R^1\pi_*\mathcal{O}_{\mathcal A}\otimes\Omega_{\Hr/\CC},
\end{equation*}
is given on the fiber over $Z^0$ by
\begin{equation*}
    \phi_0(\de z_{ij})|_{Z^0}=\begin{cases}
    \displaystyle\sum_{k=1}^{r/2}h(\delta_{ik})\otimes\de Z_{kj}|_{Z^0},&1\le j\le r/2,\\[6pt]
    \displaystyle\sum_{k=1}^{r/2}h(\bar\delta_{ik})\otimes\de Z_{(j-r/2)k}|_{Z^0},&r/2+1\le j\le r.
    \end{cases}
\end{equation*}
Here $h$ is the map in Section \ref{Cech}, and $\delta_{ik}\in\Hom_\ZZ(\Lambda_{Z^0},\CC)$ sends the period vector associated with $\beta\in\Gamma$ to $\beta_{ik}$; the homomorphism $\bar\delta_{ik}$ sends it to $\bar\beta_{ik}$.

Choose an admissible cover $\{\bar U_t\}_{t\in I}$ of $\mathcal A$ from its universal cover $\Hr\times\CC^{nr}$ as in Section \ref{Cech}. Conditions (1) and (2) are unchanged, and condition (3) requires that, for any $t,t'\in I$, the difference
\begin{equation*}
    U_{t'}-U_t:=\{(Z,\zeta'-\zeta):(Z,\zeta)\in U_t,\ (Z,\zeta')\in U_{t'}\}
\end{equation*}
intersects at most one connected component of $\Lambda:=\bigcup_Z\Lambda_Z$. If this component exists, denote it by $\beta_{t,t'}\in\Gamma$. When $Z\in\Hr$ varies, each fixed $\beta\in\Gamma$ defines a section of the universal covering space. Thus the connected components of $\Lambda$ are indexed by $\Gamma$.

Now we compute the image of $\de z_{ij}$ using the exact sequence
\begin{equation*}
    0\longrightarrow\pi^*\Omega_{\Hr/\CC}\longrightarrow\Omega_{\mathcal A/\CC}\longrightarrow\Omega_{\mathcal A/\Hr}\longrightarrow0.
\end{equation*}
To compute the connecting map, note first that each $\de z_{ij}\in\Omega_{\mathcal{A}/\mathbb{C}}(U_t)$ lifts the corresponding section, also denoted by $\de z_{ij}$, of $\pi_*\Omega_{\mathcal{A}/\Hr}$. Denote by $(\de z_{ij})_t$ the image of this lift under $U_t\longrightarrow\bar U_t$. Then it is sufficient to compute $(\de z_{ij})_{t'}-(\de z_{ij})_t$ on the overlap $\bar U_{t,t'}:=\bar U_t\cap\bar U_{t'}$.

Let $U_{t,t'}\subset U_t$ and $U_{t',t}\subset U_{t'}$ be the inverse images of $\bar U_{t,t'}$. The isomorphism $U_{t,t'}\to U_{t',t}$ sends $(Z,\zeta)$ to its translate by the period vector of $\beta_{t,t'}$. The pull-back of $z_{ij}$ is
\begin{equation*}
    z'_{ij}=\begin{cases}
    \displaystyle z_{ij}+\sum_{k=1}^{r/2}(\beta_{t,t'})_{ik}Z_{kj}+(\beta_{t,t'})_{i(r/2+j)},&1\le j\le r/2,\\[6pt]
    \displaystyle z_{ij}+\sum_{k=1}^{r/2}(\overline{\beta_{t,t'}})_{ik}Z_{(j-r/2)k}+(\overline{\beta_{t,t'}})_{ij},&r/2+1\le j\le r.
    \end{cases}
\end{equation*}
Consequently,
\begin{equation}\label{local check}
    \de z'_{ij}-\de z_{ij}=\begin{cases}
    \displaystyle\sum_{k=1}^{r/2}(\beta_{t,t'})_{ik}\de Z_{kj},&1\le j\le r/2,\\[6pt]
    \displaystyle\sum_{k=1}^{r/2}(\overline{\beta_{t,t'}})_{ik}\de Z_{(j-r/2)k},&r/2+1\le j\le r.
    \end{cases}
\end{equation}
This proves our claim about $\phi_0$. Indeed, $h(\delta_{ik})\in H^1(\mathcal{A}_{Z^0},\mathcal{O}_{\mathcal{A}_{Z^0}})$ is represented by a \v{C}ech cocycle whose local sections on $\bar U_{t,t'}$ are given by $\delta_{ik}(\beta_{t,t'})$. By definition, this is the $ik$-entry of $\beta_{t,t'}$, viewed as a constant function in the composition \eqref{composition}. The same discussion applies to $h(\bar\delta_{ik})$. Thus \eqref{local check} gives the asserted cohomology classes as $t,t'$ vary.

\textbf{Step 2.} We identify these cohomology classes with tangent vectors using Lemma \ref{app}. Write $Y^0=(Z^0-Z^{0*})/(2i)$, and use $z_{ij}$ for the complex linear coordinate functions on the covering vector space of $\mathcal A_{Z^0}$. The period formulas give
\begin{equation*}
    \begin{aligned}
    \delta_{ik}&=\frac{1}{2i}\sum_{l=1}^{r/2}(z_{il}-\bar z_{i(l+r/2)})((Y^0)^{-1})_{lk},\\
    \bar\delta_{ik}&=\frac{1}{2i}\sum_{l=1}^{r/2}(z_{i(l+r/2)}-\bar z_{il})((Y^0)^{-1})_{kl}.
    \end{aligned}
\end{equation*}
Because $h$ kills complex linear functions, it follows that
\begin{equation*}
    \begin{aligned}
    h(\delta_{ik})&=-\frac{1}{2i}\sum_{l=1}^{r/2}((Y^0)^{-1})_{lk}\,h(\bar z_{i(l+r/2)}),\\
    h(\bar\delta_{ik})&=-\frac{1}{2i}\sum_{l=1}^{r/2}((Y^0)^{-1})_{kl}\,h(\bar z_{il}).
    \end{aligned}
\end{equation*}
The same expressions are obtained by substituting the tangent vectors $\mu e_{i(k+r/2)}$ and $\mu e_{ik}$ into \eqref{general Riemann form}. More explicitly, its antiholomorphic parts in the second argument are respectively
\begin{equation*}
    -\frac{1}{2i}\sum_{l=1}^{r/2}((Y^0)^{-1})_{lk}\,\bar z_{i(l+r/2)},\qquad-\frac{1}{2i}\sum_{l=1}^{r/2}((Y^0)^{-1})_{kl}\,\bar z_{il}.
\end{equation*}
Here $e_{ij}$ denotes the elementary matrix viewed as a tangent vector in $\CC^{nr}$. Hence
\begin{equation*}
    h(\delta_{ik})=h(E_B(\mu e_{i(k+r/2)},\cdot)),\qquad h(\bar\delta_{ik})=h(E_B(\mu e_{ik},\cdot)).
\end{equation*}
Lemma \ref{app}, together with the complex linearity of the differential of the polarization, therefore identifies these two cohomology classes with
\begin{equation*}
    \frac{1}{2\pi i}\mu e_{i(k+r/2)},\qquad\frac{1}{2\pi i}\mu e_{ik}
\end{equation*}
in $\Lie(\mathcal A_{Z^0})$. This proves the asserted formula for $\phi$.

\textbf{Step 3.} We apply the determinant construction of Section \ref{Kodaira Spencer map section}. Pairing the Lie algebra factor with $\underline\Omega_{\mathcal A/\Hr}$ gives the morphism
\begin{equation*}
    \phi':\underline\Omega_{\mathcal A/\Hr}^{\otimes2}\longrightarrow\Omega_{\Hr/\CC}
\end{equation*}
which satisfies, for $1\le i,l\le n$ and $1\le j,k\le r/2$,
\begin{equation*}
    \begin{aligned}
    \phi'(\de z_{ij}\otimes\de z_{l(k+r/2)})&=\frac{\mu_{li}}{2\pi i}\de Z_{kj},\\
    \phi'(\de z_{i(j+r/2)}\otimes\de z_{lk})&=\frac{\mu_{li}}{2\pi i}\de Z_{jk},\\
    \phi'(\de z_{ij}\otimes\de z_{lk})&=\phi'(\de z_{i(j+r/2)}\otimes\de z_{l(k+r/2)})=0.
    \end{aligned}
\end{equation*}
This is the pairing \eqref{phi_2} after identifying $\underline\Omega_{\mathcal A^t}$ with $\underline\Omega_{\mathcal A}$ by the principal polarization. Contracting the row indices contributes $\det\mu$, and there are $r^2/4$ pairs of column indices. Applying the tensor-product determinant formula and the trivialization $t$ gives the stated expression for the same morphism $\psi$ as in Section \ref{Kodaira Spencer map section}.
\end{proof}

The same determinant calculation for signature $(p,q)$ gives exponents $q$ and $p$ on the determinants of the two Hodge summands. Thus the construction gives a power of the determinant of the entire Hodge bundle when $p=q=r/2$, as assumed here.

\subsection{Comparison of metrics}\label{final comparison}
We now prove the metric assertions in Theorems \ref{Intro main1} and \ref{Intro main2}. The Hodge bundles, the determinant morphisms, and the metrics are compatible with pull-back, so it suffices to compare them over the complex symmetric domains.

As in the proof of Theorem \ref{Explicit archimedean}, fix $Z^0\in\Hr$. By definition,
\begin{equation*}
    \left\|\bigwedge_{1\le i\le n,\,1\le j\le r}\de z_{ij}\right\|_{\Fal}^2(Z^0)=\frac{1}{(2\pi)^{nr}}\left|\int_{\mathcal A_{Z^0}}\bigwedge_{ij}(\de z_{ij}\wedge\de\bar z_{ij})\right|=\frac{1}{\pi^{nr}}\vol(\CC^{nr}/\Lambda_{Z^0}).
\end{equation*}
We compute this volume from the actual lattice $\Gamma$. If a real matrix has as columns a $\ZZ$-basis of $\Gamma$, the determinant of the pairing on that basis is its determinant squared times the Gram determinant of $E_B$. Since $E_B$ is self-dual on $\Gamma$, \eqref{general Riemann form} gives
\begin{equation*}
    \vol(\M_{n,r}(\CC)/\Gamma)=2^{-nr}(\det\mu)^r.
\end{equation*}
Next, conjugating the second coordinate in the period map turns it into $(a,b)\mapsto(aZ^0+b,aZ^{0*}+b)$. Its complex determinant has absolute value $2^{nr/2}\det(Y^0)^n$, so the absolute real Jacobian of the original period map is $2^{nr}\det(Y^0)^{2n}$. Therefore,
\begin{equation*}
    \begin{aligned}
    \vol(\CC^{nr}/\Lambda_{Z^0})&=2^{nr}\det(Y^0)^{2n}\vol(\M_{n,r}(\CC)/\Gamma)\\
    &=(\det\mu)^r\det(Y^0)^{2n}.
    \end{aligned}
\end{equation*}
As $Z^0$ varies, we obtain
\begin{equation*}
    \left\|\bigwedge_{1\le i\le n,\,1\le j\le r}\de z_{ij}\right\|_{\Fal}^{r/2}=\left(\frac{\det\mu}{\pi^n}\right)^{r^2/4}\det(Y)^{rn/2}.
\end{equation*}
On the other hand, the Petersson metric satisfies
\begin{equation*}
    \|\de\tau\|_{\Pet}^n=2^{r^2n/4}\det(Y)^{rn/2}.
\end{equation*}
The absolute value of the coefficient in Theorem \ref{Explicit archimedean} is $(\det\mu)^{r^2/4}(2\pi)^{-nr^2/4}$. Multiplying it by $\|\de\tau\|_{\Pet}^n$ gives exactly the Faltings norm above. Since $\|t\|=1$, this proves that $\psi$ is an isometry in type A when $g=1$, with its image equipped with the restriction of the $n$th power of the Petersson metric.

For general $F^+$, choose the same coordinates at every real embedding $\sigma$, with matrices $\mu_\sigma$ and $Z_\sigma$. Let $\Gamma$ now denote the entire transported homology lattice in $\prod_\sigma\M_{n,r}(\CC)$. The alternating form is the sum of the forms \eqref{general Riemann form}, so self-duality gives
\begin{equation*}
    \vol\left(\left(\prod_\sigma\M_{n,r}(\CC)\right)/\Gamma\right)=2^{-gnr}\prod_\sigma(\det\mu_\sigma)^r.
\end{equation*}
Applying the product of the period maps therefore gives
\begin{equation*}
    \vol(\CC^{gnr}/\Lambda_Z)=\prod_\sigma(\det\mu_\sigma)^r\det(Y_\sigma)^{2n}.
\end{equation*}
For the wedge $\alpha$ of all the fiber differentials, ordered by the embeddings, the preceding calculation becomes
\begin{equation*}
    \begin{aligned}
    \|\alpha\|_{\Fal}^{r/2}&=\pi^{-gnr^2/4}\prod_\sigma(\det\mu_\sigma)^{r^2/4}\det(Y_\sigma)^{rn/2},\\
    \psi(\alpha^{\otimes r/2}\otimes t^{-1})&=\prod_\sigma\left(\frac{\det\mu_\sigma}{(2\pi i)^n}\right)^{r^2/4}(\de\tau)^{\otimes n},
    \end{aligned}
\end{equation*}
where $\de\tau=\bigwedge_\sigma\de\tau_\sigma$ and $t$ is obtained from the ordered idempotent bases as before. Since $\|\de\tau\|_{\Pet}^n=2^{gnr^2/4}\prod_\sigma\det(Y_\sigma)^{rn/2}$, this proves the type A comparison for all $g$.

We finish with type C. At each real embedding, choose coordinates $\beta=(a,b)\in\M_{n,r}(\RR)^2$ such that the period map is $aZ+b$, with $Z\in\HH_r$, and
\begin{equation*}
    E_B(\beta,\beta')=\tr\bigl(\mu_\sigma^{-1}(ba'{}^t-ab'{}^t)\bigr),
\end{equation*}
where $\mu_\sigma$ is positive real diagonal. The Gram determinant is $(\det\mu_\sigma)^{-2r}$, and the real Jacobian of the period map is $\det(Y_\sigma)^n$. Self-duality of the entire homology lattice and the definition of the Faltings metric give
\begin{equation*}
    \|\alpha\|_{\Fal}^2=\pi^{-gnr}\prod_\sigma(\det\mu_\sigma)^r\det(Y_\sigma)^n.
\end{equation*}
The same \v{C}ech calculation gives $\phi'(\de z_{ij}\otimes\de z_{lk})=(\mu_\sigma)_{li}(2\pi i)^{-1}\de Z_{jk}$. Here $\de Z_{jk}=\de Z_{kj}$ is one coordinate for $j\ne k$, and the quotient definition of $\operatorname{Sym}^2$ uses the basis $e_je_k$ for $j\le k$. Thus there is no further factor on the off-diagonal coordinates, and taking determinants gives
\begin{equation*}
    \psi(\alpha^{\otimes(r+1)}\otimes t^{-1})=\prod_\sigma\left(\frac{\det\mu_\sigma}{(2\pi i)^n}\right)^{r(r+1)/2}(\de\tau)^{\otimes n}.
\end{equation*}
The Petersson norm is $\|\de\tau\|_{\Pet}^n=2^{gnr(r+1)/2}\prod_\sigma\det(Y_\sigma)^{n(r+1)/2}$. Multiplying by the absolute value of the coefficient in the preceding formula gives $\|\alpha\|_{\Fal}^{r+1}$, so this determinant morphism is an isometry as well. In both cases the metric on the domain is the tensor-product metric on the indicated power of $\underline\omega_{\mathcal A}$ and $\mathcal C^{-1}$. The canonical trivializations of $\mathcal C^{\otimes2}$ induced by the trace pairings are isometries, so the comparison also holds for the tensor squares of these morphisms, and for the original powers whenever $\mathcal C$ is canonically trivialized.

\subsection{Heights in the proper case}\label{subsec proper height comparison}
In this subsection we assume that $\mathcal X_K$ is proper over $\ZZn$, hence projective. We now prove the arithmetic intersection formulas stated in Section \ref{Application to modular heights}. Roughly speaking, we extend the line bundles to a projective model and use the integral and metric comparisons established above.

In type C, the generic isomorphism $\psi$ identifies
\begin{equation*}
    (\WU^{\otimes n})^{\vee\vee}\otimes\underline\omega_{\mathcal A}^{\otimes(-r-1)}\otimes\mathcal C
\end{equation*}
with a coherent fractional ideal containing $\mathcal O_{\mathcal X_K}$. Its normalized blow-up $f:\mathcal Y\to\mathcal X_K$ is projective and makes this ideal invertible; see \cite[Tags 01OF and 0806]{Stacks}. Since the ideal is already invertible at codimension-one points, $f$ is an isomorphism near those points and over the generic fiber. Therefore
\begin{equation*}
    \widetilde\omega=f^*\bigl((\WU^{\otimes n})^{\vee\vee}\bigr)/\operatorname{torsion}
\end{equation*}
is a line bundle, and $\psi$ defines an effective vertical Cartier divisor $\widetilde D$ on $\mathcal Y$. Its coefficients of strict transform are the $e_Z$ in \eqref{vertical divisor D}; its exceptional coefficients are determined by the extended fractional ideal. Note that in type A we use $r/2$ instead of $r+1$.

After a projective closure and the same blow-up construction over primes dividing $N$, extend the two line bundles to a normal projective flat model over $\mathcal O_{E_0}$. We continue to write $\mathcal Y$ for this model and $\underline\omega_{\mathcal A}$ for the extended pull-back of the Hodge line bundle. For $\mathcal C$ we use the extension in Section \ref{Application to modular heights}, with degree \eqref{degree of constant determinant bundle}. These choices fix the metrics at finite places. Note that the divisor of $\psi$ which we still denoted $\widetilde D$ may have negative coefficients at primes dividing $N$.

Equip $\widetilde\omega$ with the $n$th power of the Petersson metric and $\underline\omega_{\mathcal A}$ with the Faltings metric. The isometry gives
\begin{equation}\label{eq height full identity}
    \overline{\widetilde\omega}=(r+1)\overline{\underline\omega}_{\mathcal A}-\overline{\mathcal C}+\overline{\mathcal O}(\widetilde D)
\end{equation}
in additive notation, with $r/2$ in place of $r+1$ for type A. The canonical rational section of $\overline{\mathcal O}(\widetilde D)$ has norm one at infinity. As in Section \ref{Application to modular heights}, $h_\omega$ uses the $n$th root of the model metric of $\widetilde\omega$.

We use arithmetic degrees of metrized Deligne pairings, with their multilinearity and projection formula; see \cite[Sections 4.2 and 4.4.2, Proposition 4.1.1, Theorem 4.1.3, and Lemma 4.6.1]{YZ3}. These intersections are unchanged on common projective modifications. The generic Hodge line bundle is ample, so the height denominators are positive.

\begin{corollary}\label{cor proper height exact}
For these models and metrics, put $d=\dim X_K$. In type C one has
\begin{equation}\label{eq proper height exact}
    \begin{aligned}
    n h_\omega(X_K)-(r+1)h_{\underline\omega_{\mathcal A}}(X_K)
    ={}&-\frac{\widehat{\deg}(\overline{\mathcal C})}{[E_0:\QQ]}\\
    &+\frac{\displaystyle\sum_{\mathfrak p}\log\operatorname{Nm}(\mathfrak p)\sum_{j=0}^d\deg\bigl(\widetilde D_{\mathfrak p}\cdot c_1(\widetilde\omega)^j c_1(\underline\omega_{\mathcal A}^{\otimes(r+1)})^{d-j}\bigr)}{(d+1)[E_0:\QQ]\deg(c_1(\widetilde\omega_{E_0})^d)}.
    \end{aligned}
\end{equation}
In type A the formula holds with $r+1$ replaced by $r/2$. Here $\widetilde D_{\mathfrak p}$ is the part of $\widetilde D$ in the fiber at the finite prime $\mathfrak p$ of $E_0$, and the intersection degree is taken over its residue field. Note that only finitely many terms are nonzero.
\end{corollary}

\begin{proof}
The proof is straightforward. We apply the factorization of a difference of $(d+1)$st powers to \eqref{eq height full identity}. Each generic mixed intersection has degree $\deg(c_1(\widetilde\omega_{E_0})^d)$, since $\widetilde\omega_{E_0}$ and $\underline\omega_{\mathcal A,E_0}^{\otimes(r+1)}$ differ by the constant line bundle. Therefore, the projection formula gives the term involving $\widehat{\deg}(\overline{\mathcal C})$. Since the canonical rational section of $\overline{\mathcal O}(\widetilde D)$ has norm one at infinity, its intersection is the finite sum in \eqref{eq proper height exact}. Finally, $\deg(c_1(\widetilde\omega_{E_0})^d)=(r+1)^d\deg(c_1(\underline\omega_{\mathcal A,E_0})^d)$. Dividing by $(d+1)[E_0:\QQ]\deg(c_1(\widetilde\omega_{E_0})^d)$ gives the stated normalization. The same argument with $r/2$ proves type A.
\end{proof}

Note that when $\widetilde D$ comes from a finite divisor on $\operatorname{Spec}\mathcal O_{E_0}$, the projection formula reduces the finite sum to its normalized arithmetic degree. In particular, $\widetilde D=\operatorname{div}(c)$ for $c\in\QQ_{>0}$ gives $\log c$, as in Section \ref{Application to modular heights}.

\noindent\small{School of Mathematical Sciences, Peking University, Beijing 100871, China}

\noindent\small{\it Email: ziqiguo0603@pku.edu.cn}

\end{document}